\documentclass[10pt,a4paper]{amsart}

\usepackage[a4paper,top=3.5 cm,bottom=3 cm,left=2.5 cm,right=2.5 cm]{geometry}
\allowdisplaybreaks
\usepackage[T1]{fontenc}
\usepackage[english]{babel}
\usepackage{url}
\usepackage{amsmath,amssymb,amsthm}
\usepackage{mathtools}
\usepackage{graphicx}
\usepackage{xcolor} 
\usepackage{bbm}
\usepackage{dsfont}
\usepackage{stmaryrd}
\usepackage{esint}
\usepackage{floatrow}
\usepackage{epsfig}
\usepackage{hyperref}
\usepackage[normalem]{ulem}
\usepackage{mathrsfs}
\usepackage{tikz}
\usetikzlibrary{decorations.markings,backgrounds}
\usetikzlibrary{arrows.meta}
\usepackage{pgfplots}\pgfplotsset{compat=1.18}
\usepackage{pst-plot}
\usepackage{subfig}
\usepackage{caption}
\usepackage{float}
\usepackage[colorinlistoftodos]{todonotes}
\usepackage{import}
\usepackage{enumitem}
\usepackage{verbatim,fancyvrb}
\usepackage{epigraph}

\definecolor{mblue}{rgb}{         0    0.4470    0.7410}
\definecolor{mred}{rgb}{    0.8500    0.3250    0.0980}  
\definecolor{myellow}{rgb}{    0.9290    0.6940    0.1250}

\theoremstyle{plain}
\newtheorem{lemma}{Lemma}[section]
\newtheorem{proposition}[lemma]{Proposition}
\newtheorem{theorem}[lemma]{Theorem}

\newtheorem*{idea*}{Idea}

\theoremstyle{definition}
\newtheorem{deff}[lemma]{Definition}

\theoremstyle{remark}
\newtheorem{remark}[lemma]{Remark}

\numberwithin{equation}{section}

\newcommand{\loc}{\mathrm{loc}}
\newcommand{\rel}{\mathrm{rel}}
\newcommand{\lip}{\mathrm{Lip}}
\newcommand{\R}{\mathbb{R}}
\newcommand{\N}{\mathbb{N}}

\newcommand{\eps}{{\varepsilon}}

\DeclareMathOperator{\dive}{div}

\newcommand{\abs}[1]{\left\lvert#1\right\rvert}
\newcommand{\norm}[1]{\left\lVert#1\right\rVert}
\newcommand\skal[1]{\left\langle#1\right\rangle}

\renewcommand{\div}{\operatorname{div}}

\newcommand{\dd}{\,\mathrm{d}}

\newcounter{eq}[section]

\title[Weak-strong uniqueness]{Relative energy method for weak-strong uniqueness of the inhomogeneous Navier-Stokes equations}

\subjclass[2020]{76D03; 76D05}
\keywords{Inhomogeneous Navier-Stokes equations, weak-strong uniqueness, relative energy method, Leray-Hopf weak solutions
\\$^*$corresponding author: stefan.skondric@fau.de}

\date{}

\author[T. Crin-Barat]{Timothée Crin-Barat}
\address[T. Crin-Barat]{Chair for Dynamics, Control, Machine Learning and Numerics, Alexander Von Humboldt- Professorship, Department of
Mathematics, Friedrich-Alexander-Universität Erlangen-Nürnberg, Cauerstraße 11, 91058 Erlangen, Germany.}
\email{timothee.crin-barat@fau.de}

\author[S. \v{S}kondri\'c]{Stefan \v{S}kondri\'c$^*$}
\address[S. \v{S}kondri\'c]{Chair of Analysis, Department of
Mathematics, Friedrich-Alexander-Universität Erlangen-Nürnberg, Cauerstraße 11, 91058 Erlangen, Germany.}
\email{stefan.skondric@fau.de}

\author[A. Violini]{Alessandro Violini}
\address[A. Violini]{Department Mathematik und Informatik, Universität Basel, Spiegelgasse 1, 4051
Basel, Switzerland}
\email{alessandro.violini@unibas.ch}

\begin{document}

\begin{abstract}
We present a weak-strong uniqueness result for the inhomogeneous Navier-Stokes (INS) equations in $\R^d$ ($d=2,3$) for bounded initial densities that are far from vacuum. Given a strong solution within the class employed in \cite{PaicuZhangZhang2013} and \cite{ChenZhangZhao2016} and a \textit{Leray-Hopf} weak solution, we establish that they coincide if the initial data agree. The strategy of our proof is based on the relative energy method and new $W^{-1,p}$-type stability estimates for the density. A key point lies in proving that every Leray-Hopf weak solution originating from initial densities far from vacuum remains distant from vacuum at all times.
\end{abstract}

\maketitle

\section{Introduction and main results}
\subsection{Presentation of the model and literature}
We consider the \emph{inhomogeneous incompressible Navier--Stokes equations} in $\R^d$ for $d=2,3$:
\begin{align}\label{eq:ns}
  \begin{cases}
    \partial_t \rho + \dive (\rho u) = 0, & t > 0, \ x \in \R^d, \\
    \partial_t (\rho u) + \dive (\rho u \otimes u) - \nu \Delta u + \nabla P = 0, & t > 0, \ x \in \R^d, \\
    \operatorname{div} u = 0, & t > 0, \ x \in \R^d, \\
    \rho(0,x) = \rho_0(x), & x \in \R^d,\\
    u(0,x) = u_0(x), & x \in \R^d,
  \end{cases}
\end{align} 
where $\rho=\rho(t,x) \geq 0$ is a scalar \emph{density function}, $P=P(t,x)\in \R$ is the \emph{pressure} and 
$u=u(t,x)\in\R^d$ is the \emph{velocity field} of the fluid. The first equation of \eqref{eq:ns} represents the 
conservation of mass, the second is the momentum equation and the third is the incompressibility condition. 

The inhomogeneous Navier-Stokes system shares many properties with its homogeneous counterpart (\eqref{eq:ns} with 
$\rho \equiv 1$). Formally, its solutions satisfy the energy balance:
\begin{align}\label{eq:energyconsintro}
    \frac{1}{2} \int_{\R^d} \rho(t) |u(t)|^2 \dd x + \nu \int_0^t \int_{\R^d} |\nabla u|^2 \dd x \dd s =
    \frac{1}{2} \int_{\R^d} \rho_0 |u_0|^2 \dd x,
\end{align}
and they obey the same scaling symmetries as the solutions of the homogeneous
Navier-Stokes equations. More precisely, if $(\rho,u,P)$ is a solution of the inhomogeneous Navier-Stokes equations, then,
for every $\lambda >0$, we have that $(\rho_\lambda,u_\lambda,P_\lambda)$, given by
\begin{align*}  (\rho_\lambda(t,x),u_\lambda(t,x),P_\lambda(t,x)) = 
    (\rho(\lambda^2 t,\lambda x),\lambda u(\lambda^2 t,\lambda x),
    \lambda^2 P(\lambda^2 t,\lambda x)), \quad (t,x) \in (0,\infty) \times \R^d,
\end{align*}
is still a solution of the inhomogeneous Navier-Stokes equations. 

These fundamental properties have significant implications for the theory of inhomogeneous Navier-Stokes equations. Specifically, the energy equality
\eqref{eq:energyconsintro} suggests seeking weak solutions within a framework similar to that of the homogeneous Navier-Stokes 
equations. In that direction, one of the first existence result of global-in-time weak solutions, satisfying \eqref{eq:energyconsintro}, 
was obtained by Kazhikhov in \cite{Kazikov1974} for densities bounded from below. Later, Simon \cite{Simon1990} removed the lower bound on the density and Lions and Desjardins \cite{Lions1996,Desjardins19972} constructed weak solutions for the density-dependent Navier-Stokes equations in arbitrary dimensions. The uniqueness of these solutions remains, however, an outstanding open problem, even in two dimensions.
\medbreak
In \cite{LadyzhenskayaSolonnikov1975}, for $C^1$ initial densities and smooth enough initial velocities, Ladyzhenskaya and Solonnikov 
established the first well-posedness result for the inhomogeneous Navier-Stokes equations in smooth bounded domains of $\R^2$ and $\R^3$. This was
generalized to weakly differential initial densities by Desjardins \cite{Desjardins19971} and Danchin \cite{Danchin2006}. 

In the present paper, we focus on initial densities that are only assumed to be bounded and that are far from vacuum. Under this condition, critical spaces (for the initial velocity) provide natural frameworks to establish well-posedness results. As mentioned, the inhomogeneous Navier-Stokes equations have the same scaling symmetries as the homogeneous Navier-Stokes
equations and, therefore, the same critical spaces for the velocity, e.g.
\begin{align}\label{eq:criticalspaces}
    \dot{H}^{d/2-1}(\R^d), \dot{B}^{d/p-1}_{p,q}(\R^d), \dot{B}^{d/p-1}_{p,\infty}(\R^d),
\end{align}
with $p\in (1,\infty)$ and $q\in [1,\infty)$. See \cite{AlbrittonPhDThesis} for a review on the topic of critical frameworks for the homogeneous Navier-Stokes equations. The uniqueness of solutions emanating from initial velocities in critical spaces of the inhomogeneous Navier-Stokes equations is, so far, mostly unknown. One of the reasons is the 
lack of Lipschitz estimates for the solutions of \eqref{eq:ns} with initial values in \eqref{eq:criticalspaces}. More precisely,
let $X$ be one of the spaces in \eqref{eq:criticalspaces} and let $(\rho,u)$ be a solution of \eqref{eq:ns} with respect to
the initial value $(\rho_0,u_0)$, with $u_0 \in X$. Then, for $q\ne1$, it is not clear how to show that the solutions of \eqref{eq:ns} satisfy, for any $T>0$,
\begin{align}\label{eq:lipestimate}
    \int_0^T \norm{\nabla u}_\infty \dd t \leq \norm{u_0}_X.
\end{align}
Note that even the solutions of the heat equation do not satisfy such estimates in these critical frameworks. For the heat equation, \eqref{eq:lipestimate} becomes true if $X=H^s(\R^d)$ for $s>d/2-1$, or $X=\dot{B}^{d/p-1}_{p,1}(\R^d)$.

This observation is consistent with the results obtained by Paicu, Zhang and Zhang \cite{PaicuZhangZhang2013} and Chen, Zhang 
and Zhao \cite{ChenZhangZhao2016}. In \cite{PaicuZhangZhang2013}, the authors proved the existence and uniqueness of solutions of \eqref{eq:ns} when $u_0 \in H^\eta(\R^2)$, with $\eta>0,$ and $\rho_0 \in L^\infty(\R^2)$ is bounded from below, i.e. far from vacuum. Furthermore, in the three-dimensional case, the authors proved the existence and uniqueness of global-in-time solutions of \eqref{eq:ns}
provided the initial value $u_0$ belongs to  $H^1(\R^3)$ and satisfies the smallness condition
\begin{align}
    \norm{u_0}_2^{1/2} \norm{\nabla u_0}_2^{1/2} \leq \eps_*,
\end{align} for $\eps_*>0$ a small enough constant. Without the smallness assumption, their construction provides local-in-time unique solutions for $(\rho_0,u_0)\in L^\infty(\R^3)\times H^{1}(\R^3)$. In \cite{ChenZhangZhao2016}, the authors generalized this statement and obtained a global-in-time well-posedness result for initial velocities lying in $H^\eta(\R^3)$ with $\eta>1/2$ and under the smallness condition
\begin{align}
    \norm{u_0}_{\dot{H}^{1/2}} \leq \eps_*.
\end{align}
More recently, Danchin and Wang \cite{DanchinWang2022} showed the global-in-time well-posedness for initial velocities in the critical Besov space $\dot{B}^{\frac{d}{p}-1}_{p,1}(\R^d)$ with $p\in (1,d)$, and for small and
only bounded initial densities. In \cite[Theorem 1.7]{DanchinWang2022}, they suggest a way to derive weak-strong uniqueness results in $2d$ far from vacuum and in $3d$ including vacuum. This is strongly related to our goal in the present paper, we discuss their result in more detail in Remark \ref{rq:DW}.
Then, Danchin \cite{Danchin2023} justified the global 
well-posedness when the initial density is in $L^\infty$ and the initial velocity lies in a subspace of $L^2$ obtained by interpolation. In the case of smooth densities, the space defined in \cite{Danchin2023} coincides with the critical space $B^{0}_{2,1}(\R^2)$, otherwise, its construction depends on $\rho_0$.
We also mention \cite{DanchinVasilyev23} where Danchin and Vasilyev
obtain a well-posedness result in critical tent spaces where the initial velocities lie in a subspace of BMO$^{-1}$.

In the presence of initial vacuum, i.e. $\rho_0\geq0$, the local and global well-posedness in bounded domains was justified by Danchin and Mucha \cite{DanachinMucha2019}. Recently, in \cite{PrangeTan2023}, Prange and Tan studied the existence and uniqueness of solutions on $\R^d$, $d=2,3,$ with some 
special vacuum configurations. In \cite[Proposition 4.1]{PrangeTan2023}, they established a result concerning 
weak-strong uniqueness in $\mathbb{R}^2$, refer to Remark \ref{rq:RQ1} to see how this compares to ours.

\smallbreak
Under regularity assumptions on the initial density, Danchin \cite{Danchin2003} demonstrated the existence and uniqueness of local-in-time solutions if
\begin{align*}
    \norm{\rho_0 - 1}_{\dot{B}^{d/2}_{2,1}} \leq \eps_*\quad \text{and} \quad u_0 \in \dot{B}^{d/2-1}_{2,1}(\R^d),
\end{align*}
and the global-in-time existence if the initial velocity is sufficiently small in $\dot{B}^{d/2-1}_{2,1}(\R^d)$. This result was extended by Abidi and Paicu \cite{AbidiPaicu2007} and Danchin and Mucha \cite{DanchinMucha2012}, for 
suitable initial densities, to the $L^p$ framework:
\begin{align*}
    u_0 \in \dot{B}^{d/p-1}_{p,1}(\R^d),
\end{align*}
for $1 < p < d$ and $1 \leq p < 2d,$ respectively. This was further generalized by Haspot \cite{Haspot2012} to cover 
the full range $p \in (1, \infty)$. In these results, the philosophy is to balance the regularity between the initial 
velocity and the initial density, see \cite[Theorem 1.1]{Haspot2012}. Furthermore, Haspot removed the smallness 
assumption in the case $p=2,$ see \cite[Theorem 1.2]{Haspot2012}. Then, relying on $L^1$-maximal regularity estimates, 
Xu \cite{Xu2022} removed the smallness condition in the case $d=3.$
In 2D,  Abidi and Gui \cite{AbidiGui2021} established the well-posedness for $\rho_0\in\dot{B}^1_{2,1}(\R^2)$ and $u_0\in\dot{B}^0_{2,1}(\R^2) $, lowering the 
regularity on the initial density. Recently, Abidi, Gui and Zhang \cite{AbidiGuiZhang2023} generalized this result following the balance of regularity philosophy. 

\subsection{Aims of the paper and notions of solution}
In this paper, we prove a weak-strong uniqueness result for the inhomogeneous Navier-Stokes equations in both 
$\mathbb{R}^2$ and $\mathbb{R}^3$ for initial densities in $L^\infty$ that do not exhibit vacuum. We consider initial data satisfying
\begin{equation}\label{eq:initialreg}
\begin{aligned} 
    & 0 < c_0 \leq \rho_0(x) \leq C_0 < +\infty \quad \text{for almost every } x \in \R^d,\\
    &u_0 \in  L^2_\sigma(\R^d; \R^d),
\end{aligned}
\end{equation}
for two positive constants $c_0$ and $C_0$ and where $L^2_\sigma$ denotes the space of divergence-free $L^2$ vector fields. The weak solutions we consider are the energy-admissible weak solutions or \textit{Leray-Hopf} weak solutions defined by Lions in \cite{Lions1996}.

\begin{deff}[\textit{Leray-Hopf} weak solution]\label{def:lerayhopfsol}
We designate a pair $(\rho,u)$ as a \textit{Leray-Hopf} weak solution of \eqref{eq:ns} with initial data $(\rho_0,u_0)$ 
satisfying \eqref{eq:initialreg} if
    \begin{itemize}
        \item[(i)] The solution satisfies
        \begin{align*}
           & \sqrt{\rho} u \in L^\infty((0,\infty);L^2(\R^d)),
            \quad \rho \in L^\infty((0,\infty)\times \R^d),
            \\ &  u \in L^2_{\loc}((0,\infty)\times\R^d))  \quad \text{and} \quad \nabla u \in L^2((0,\infty);L^2(\R^d)). 
        \end{align*}
        \item[(ii)] The first three equations of \eqref{eq:ns} are satisfied in the sense of distributions:
        \smallbreak
        \begin{itemize}
            \item[$\bullet$] For every  $\varphi \in C_c^\infty([0,\infty) \times \R^d; \R)$,
        \begin{align*}
            - \int_0^\infty \int_{\R^d} \rho \partial_s \varphi \dd x \dd s
            =  \int_0^\infty \int_{\R^d} \rho u \cdot \nabla \varphi \dd x \dd s 
            + \int_{\R^d} \rho_0 \varphi(0) \dd x,
        \end{align*}
        \item[$\bullet$] For every $\varphi \in C_{c,\sigma}^\infty([0,\infty) \times \R^d;\R^d)$,
        \begin{align*}
            - \int_0^\infty \int_{\R^d} &\rho \partial_t \varphi \cdot u+
            \rho u \otimes u : \nabla \varphi \dd x \dd s  
            + \nu \int_0^\infty \int_{\R^d} \nabla u : \nabla \varphi \dd x \dd s = \int_{\R^d} \rho_0 u_0 \cdot \varphi(0)\dd x,
        \end{align*}
        \item[$\bullet$]For every  $\varphi \in C_c^\infty([0,\infty) \times\R^d;\R)$,
        \begin{equation*}
            \int_0^\infty\int_{\R^d} u\cdot \nabla \varphi \dd x\dd s =0.
        \end{equation*}
  \end{itemize}
        \item[(iii)] For every $t \in (0,\infty)$, the energy inequality holds:
        \begin{align}\label{eq:EE}
            \int_{\R^d} \rho(t) |u(t)|^2 \dd x + \nu \int_0^t \int_{\R^d} |\nabla u(s)|^2 \dd x \dd s \leq 
            \int_{\R^d} \rho_0|u_0|^2 \dd x.
        \end{align}
    \end{itemize}
\end{deff}

In \cite{CrippaPhDThesis}, Crippa demonstrated that weak solutions of transport equations with divergence-free 
$L^1_{\loc}$ velocity field can be modified on a time-set of measure zero such that it belongs to $C_{w^*}([0,\infty);
L^\infty(\R^d))$. Based on this observation, we assume that the density of a Leray-Hopf weak solution satisfies $\rho \in C_{w^*}
([0,\infty); L^\infty(\R^d))$ in the rest of the manuscript.

Next, we recall the following existence result for weak solutions, see \cite[Theorem 2.1]{Lions1996}.

\begin{theorem}[\cite{Lions1996}]\label{thm:existweaksol}
    Let $(\rho_0,u_0)$ satisfy \eqref{eq:initialreg}. There exists a global weak solution of \eqref{eq:ns} with 
    initial data $(\rho_0,u_0)$ satisfying the energy inequality \eqref{eq:EE}. Furthermore, for all $0 \leq c_0 \leq C_0 < 
    \infty$, the Lebesgue measure of
    \begin{align}\label{eq:measuredensity}
        \{ x \in \R^d\, |\, c_0 \leq \rho(t,x) \leq C_0 \}
    \end{align}
    is independent of $t.$
\end{theorem}

\smallbreak
Regarding the notion of strong solutions, we rely on the results obtained 
by Paicu, Zhang and Zhang in \cite[Theorem 1.1]{PaicuZhangZhang2013} and Chen, Zhang and Zhao in \cite{ChenZhangZhao2016}. Before recalling their results, we introduce the 
following quantities: for $\eta> d/2-1$ and $\sigma(t)=\min\{1,t\}$,
\begin{align*}
    & A^{\eta}_1(t;\rho, u) :=\frac{1}{2} \sup_{s \in [0,t]} \sigma(s)^{1-\eta}\int_{\R^d} |\nabla u(s,x)|^2 \dd x +
    \int_0^t \sigma(s)^{1-\eta} \int_{\R^d} \rho |\partial_t u|^2+|\nabla^2 u|^2 + |\nabla P|^2 \dd x \dd s,\\
    & A^\eta_2(t;\rho,u) := \frac{1}{2} \sup_{s \in [0,t]} \sigma(s)^{2-\eta} \int_{\R^d} \rho |\partial_t u|^2+|\nabla^2 u|^2 
    + |\nabla P|^2 \dd x + \int_0^t\sigma(s)^{2-\eta}\int_{\R^d}  |\partial_t \nabla u|^2 \dd x \dd s,
\end{align*}
and 
\begin{equation*}
    \mathcal{S}^\eta(0,T) := \left\{ (\rho, u) \; \middle| \; \text{ for every } t \in [0, T), \text{ there exists a $C>0$ s.t. } A_1^{\eta}(t; \rho, u), A_2^{\eta}(t; \rho, u) \leq C \right\}.
\end{equation*}

Combining the results obtained in \cite{PaicuZhangZhang2013} and \cite{ChenZhangZhao2016} leads to the following statement.
\begin{theorem}[\cite{PaicuZhangZhang2013,ChenZhangZhao2016}]\label{thm:PZZ}
   Let $(\rho_0,u_0)$ as in \eqref{eq:initialreg}.
    \begin{enumerate}
        \item Let $d=2$ and $\eta > 0$. For $u_0 \in H^\eta(\R^2)$, the system \eqref{eq:ns} admits a unique global-in-time solution $(\rho,u)
        \in \mathcal{S}^\eta(0,\infty).$ 
        \item Let $d=3$ and $\eta=1$. For $u_0 \in H^1(\R^3)$, the system \eqref{eq:ns} admits a unique 
        local-in-time solution $(\rho,u) \in \mathcal{S}^1(0,T^*)$ for the maximal time $T^*=T^*(\norm{u_0}_1)>0$. Furthermore, there exists an absolute constant $\eps_*>0$
        such that, for every $u_0 \in H^1(\R^3)$ satisfying the smallness condition
        \begin{align*}
            \norm{u_0}_2^{1/2} \norm{\nabla u_0}_2^{1/2} \leq \eps_*,
        \end{align*}
        we have $T^*=\infty.$
        \item Let $d=3$ and $\eta > 1/2$. There exists an absolute constant $\eps_*>0$ such that, for every $u_0 \in 
        H^\eta(\R^3)$ satisfying
        \begin{align*}
            \norm{u_0}_{\dot{H}^{1/2}} \leq \eps_*,
        \end{align*}
       the system \eqref{eq:ns} admits a unique global-in-time solution $(\rho,u) \in \mathcal{S}^\eta(0,\infty).$
    \end{enumerate}
    In each case, the solution $(\rho,u)$ satisfies the energy equality: 
    \begin{equation*}
        A_0(t;\rho, u) :=\ \int_{\R^d} \rho(t) |u(t)|^2 \dd x + \nu \int_0^t \int_{\R^d} |\nabla u(s)|^2 \dd x \dd s 
        = \int_{\R^d} \rho_0|u_0|^2 \dd x,
    \end{equation*} for every $t\in (0,T^*)$, where $T^*$ denotes the maximal time of existence of the solution.

    Moreover, the density stays away from vacuum: there exist $c_0,C_0>0$ such that for almost every $(t,x) \in [0, T^*)\times \R^d$,
    \begin{align}\label{eq:strongwawayvacuum}
        c_0\leq \rho(t, x) \leq C_0,
    \end{align}
 and, for every $t\in [0,T^*]$, we have
    \begin{equation}\label{est:lip}
        \int_0^t \norm{\nabla u}_\infty \dd t < \infty.
    \end{equation}
\end{theorem}
Motivated by Theorem \ref{thm:PZZ}, our definition of strong solutions is as follows.
\begin{deff}[Strong solutions]\label{def:strongsol}
Let $0< T \leq \infty$ and $(\rho_0,u_0)$ satisfy \eqref{eq:initialreg}. We say that a Leray-Hopf solution $(\rho,u)$ 
of \eqref{eq:ns} is a strong solution of \eqref{eq:ns} on $(0,T)$ if there exists an $\eta > d/2-1$ such that 
$(\rho, u) \in \mathcal{S}^\eta(0,T).$ Furthermore, we define the maximal time $0<T^*\leq \infty$ of the strong solution $(\rho,u)$ by
\begin{align*}
    T^* := \sup\{ T \in (0,\infty) | (\rho, u) \in \mathcal{S}^\eta(0,T) \}.
\end{align*}
\end{deff}

\subsection{Main results}
We are now ready to state our main result, which aims to close the gap between the Leray-Hopf weak solutions 
from Definition \ref{def:lerayhopfsol} and the strong solutions from Definition \ref{def:strongsol}.

\begin{theorem}[Weak-strong uniqueness]\label{thm:weakstrong1}
    Let $d=2,3$ and $(\rho_1,u_1)$ be a \textit{Leray-Hopf} weak 
    solution of \eqref{eq:ns} associated with the initial data $(\rho_0,u_0)$ satisfying \eqref{eq:initialreg}. If there exists $(\rho_2,u_2)$ a strong solution of \eqref{eq:ns} on $(0,T^*)$ with the same initial data $(\rho_0,u_0)$, then $(\rho_2,u_2)=(\rho_1,u_1)$ for almost every $(t,x)\in [0,T^*)\times \R^d$, where $0<T^*\leq\infty$ denotes the maximal time of existence of the strong solution.
\end{theorem}

\begin{remark}\label{rq:RQ1}
Compared to previous results on weak-strong uniqueness for \eqref{eq:ns}, our result covers larger classes of strong 
solutions. More precisely, the existence of our strong solutions holds for $u_{0}\in H^{\eta}$ with $\eta>\frac{d}{2}-1$ 
whereas $\eta>2$ (for $d=3$) and $\eta\geq 1$ are needed in \cite{Germain2008} and \cite{PrangeTan2023}, respectively. 

In addition, our result is 
valid in the three-dimensional setting, whereas \cite[Proposition 4.2]{PrangeTan2023} is restricted to weak solutions 
lying additionally in critical spaces: $u_1\in L^4((0,T);L^6(\R^3))$. 
Our improvements can be attributed to the different approach we employ: the relative energy method, and the fact that we 
deal with initial densities far from vacuum, in contrast with \cite{Germain2008,PrangeTan2023} where nonnegative initial 
densities are considered. 
\end{remark}
\begin{remark}\label{rq:DW}
In \cite{DanchinWang2022}, the authors suggest a way to derive a weak-strong uniqueness result.
For the strong solutions we consider, it is not direct that they verify the conditions given in \cite[Theorem 1.7]{DanchinWang2022} to derive a weak-strong uniqueness result, see Section \ref{sec:compare} for more details.  Nevertheless, it is worth highlighting that employing \cite[Theorem 1.7]{DanchinWang2022} would result in a weak-strong uniqueness result in $3d$ allowing initial vacuum states.
\end{remark}

The absence of vacuum formation is a key point in our analysis and as it is a result of independent interest, we state it 
as follows.

\begin{theorem}[No Vacuum Formation]\label{thm:novacuum0}
    Let \((\rho_0,u_0) \in L^\infty(\R^d) \times L_\sigma^2(\R^d;\R^d)\) and assume that there exists two constants $c_0,C_0$ such that, for almost every $x\in\R^d$,  \[ 0 < c_0 \leq \rho_0(x) 
    \leq C_0. \] Let \((\rho,u)\) be a weak solution of \eqref{eq:ns} in the sense of Definition 
    \eqref{def:lerayhopfsol} with the initial data $(\rho_0,u_0)$.  Then, for almost every \((t,x) \in [0,\infty) 
    \times \R^d\), \[ 0 < c_0 \leq \rho(t,x) \leq C_0. \]
\end{theorem}

\begin{remark}
  In contrast with Theorem \ref{thm:existweaksol}, where Lions proved the absence of vacuum formation for a particular weak solution, Theorem \ref{thm:novacuum0} 
   shows that this property holds for any weak solution of \eqref{eq:ns} in the sense of Definition \ref{def:lerayhopfsol}.
\end{remark}

\subsection{Strategy of proof of Theorem \ref{thm:weakstrong1}}

Our proof is based on the relative energy method, which has multiple applications in fluid mechanics, as highlighted in \cite{Wiedemann2018} by Wiedemann. For applications to hyperbolic conservation laws see \cite{Dafermos2009} by Dafermos. We also borrow tools from the works of Li \cite{Li2017} and Danchin and Wang \cite{DanchinWang2022}. Let us briefly describe our approach. Let $(\rho_1,u_1,P_1)$ and $(\rho_2,u_2,P_2)$ be two regular solutions of \eqref{eq:ns} with respect to the same initial data $(\rho_0,u_0)$. While we assume the solutions to be regular enough here, a major part of our proof consists in showing, with approximation arguments, that the computations hold when $(\rho_1,u_1)$ is a weak solution and $(\rho_2,u_2)$ a strong solution. We define
\begin{align*}
    \delta \rho = \rho_1 - \rho_2, \quad \delta u = u_1 - u_2 \quad \textrm{ and}\quad \delta P =P_1 - P_2.
\end{align*}
The error unknown $(\delta \rho , \delta u, \delta P)$ satisfies the following equations 
\begin{align}\label{eq:nsdiffintro}
  \begin{cases}
   \partial_t \delta \rho + \dive( \delta \rho u_2) = - \dive (\rho_1 \delta u),\\
    \rho_1 \left( \partial_t \delta u + (u_1 \cdot \nabla) \delta u\right) - \Delta \delta u + \nabla \delta P =
    -\delta \rho \dot{u}_2 - \rho_1 (\delta u \cdot \nabla) u_2,\\
    \delta \rho(0)=0,\, \delta u(0) = 0,
    \end{cases}
\end{align}
where we used the material derivative notation
\begin{align*}
    \dot{u}_2 := \partial_t u_2 + (u_2 \cdot \nabla) u_2.
\end{align*}
Our proof is based on the relative energy method which boils down to bounding the relative energy:
\begin{align*}
    E_\rel(t):= \frac{1}{2} \int_{\R^d} \rho_1(t) \abs{\delta u(t)}^2 \dd x, \quad t>0.
\end{align*}
Multiplying the second equation of \eqref{eq:nsdiffintro} by $\delta u,$ for sufficiently regular solutions, and 
integrating in time and space, we obtain
\begin{align} \label{eq:RE}
    E_\rel(t) + \nu \int_0^t \int_{\R^d} |\nabla \delta u|^2 \dd x \dd s &= - \int_0^t \int_{\R^d} \delta \rho \dot{u}_2
    \cdot \delta u \dd x \dd s - \int_0^t \int_{\R^d} \rho_1 (\delta u \cdot \nabla) u_2 \cdot \delta u
    \dd x \dd s\\
    &=- \int_0^t \int_{\R^d} \delta \rho \dot{u}_2 \cdot \delta u \dd x \dd s 
    - \int_0^t \int_{\R^d} \rho_1 \delta u \otimes \delta u : \nabla u_2 \dd x \dd s. \nonumber
\end{align}
We emphasize here that obtaining \eqref{eq:RE} in our weak-strong setting is not straightforward. The difficulty is that it is not clear how to bound
\begin{align*}
    \partial_t u_1 \quad\text{and}\quad \partial_t (\rho_1 u_1),
\end{align*} in order to apply Lions–Magenes lemma.
In particular, applying the Leray projector to the equation only allows to recover
\begin{align*}
    \partial_t \mathbb{P}(\rho_1 u_1) \in L^{4/d}(0,T;\dot{H}^{-1}(\R^d)).
\end{align*}
Therefore, in our weak-strong framework, one cannot justify the following computations, which were used in \cite{DanchinWang2022,Li2017} to derive \eqref{eq:RE}, 
\begin{align*}
    \int_0^t \int_{\R^d} \rho_1 \partial_t \delta u \cdot \delta u \dd x \dd s 
    &= \int_0^t \int_{\R^d} \rho_1 \partial_t \frac{1}{2}|\delta u|^2 \dd x \dd s\\
    &= \int_{\R^d} \rho_1(t) \frac{1}{2}|\delta u(t)|^2 \dd x - \int_0^t \int_{\R^d} \partial_t \rho_1 
    \frac{1}{2}|\delta u|^2 \dd x \dd s\\
    &= \int_{\R^d} \rho_1(t) \frac{1}{2}|\delta u(t)|^2 \dd x - \int_0^t \int_{\R^d} \rho_1 
    (u_1 \cdot \nabla) \delta u \cdot \delta u \dd x \dd s.
\end{align*}
In Lemma \ref{lem:badgronwall}, we show that \eqref{eq:RE} holds for $(\rho_1,u_1)$ a Leray-Hopf solution, avoiding the above computations. To do so, employing approximation arguments, we establish in Section \ref{sec:admin} that the strong solutions are admissible test functions for both the weak formulation of the transport and momentum equation.

Once \eqref{eq:RE} at hand, we aim to control the right-hand side terms to apply \textit{Grönwall's inequality} and conclude $E_\rel \equiv 0.$ 
Under the condition that
\begin{align}\label{eq:lipschitzu2}
    \nabla u_2 \in L^1_\loc(0,\infty;L^\infty(\R^d)),
\end{align}
the estimate of the second r.h.s. of \eqref{eq:RE} is straightforward and we have
\begin{align*}
    - \int_0^t \int_{\R^d} \rho_1 \delta u \otimes \delta u : \nabla u_2 \dd x \dd s \leq
    \int_0^t \norm{\nabla u_2(t)}_\infty E_\rel(s) \dd s.
\end{align*}
Controlling the term 
\begin{align}\label{eq:productintro}
    \int_0^t \int_{\R^d} \delta \rho \dot{u}_2 \cdot \delta u \dd x \dd s
\end{align}
is more difficult as it is not \textit{quadratic in $\sqrt{\rho_1} \delta u$}, which is problematic to apply a Grönwall argument. 

To control this term, we derive stability estimates for $\delta \rho$ in well-chosen norms. The key observation is that the strong solution $u_2$ satisfies the Lipschitz 
bound \eqref{eq:lipschitzu2}, which implies that there exists a flow $X \colon [0,\infty) \times \R^d 
\to \R^d$ generated by $u_2$. Assuming that $\rho_1 \delta u$ is smooth, from the first equation of 
\eqref{eq:nsdiffintro}, we have the following representation formula, see \cite[Proposition 6]{AmbrosioCrippa2014},
\begin{align}\label{eq:formulatransintro}
    \delta \rho(t,X(t,x)) = - \int_0^t \dive\left((\rho_1 \delta u)(\tau,X(\tau,x)
    \right) \dd \tau.
\end{align}
Using \eqref{eq:formulatransintro}, approximation arguments
 and commutator estimates, for $p,q>1$ such that $1/p+1/q=1$, we bound \eqref{eq:productintro} as
\begin{align}\label{eq:productestimateintro}
    - \int_0^t \int_{\R^d} \delta \rho \dot{u}_2 \cdot \delta u \dd x \dd s &\leq  C \int_0^t \int_0^s \int_{\R^d}(\dive\left(\rho_1 \delta u))(\tau,X(\tau,x)
    \right)\dot{u}_2\cdot \delta u  \dd\tau \dd x\dd s
    \\&\leq 
    C \int_0^t \int_0^s \norm{(\rho_1 \delta u)(\tau)}_p \dd \tau \norm{\nabla(\dot{u}_2(s) \cdot \delta u(s))}_q \dd s, \nonumber
\end{align}
where $C$ is a constant depending on the flow and its derivatives. In some sense, \eqref{eq:productestimateintro} can be understood as a $W^{-1,p}$ stability estimate for $\delta\rho$.

Then, using \textit{Gagliardo-Nirenberg inequality}, we derive $L^p$ bounds for $\rho_1 \delta u$, which ensures that
\eqref{eq:productintro} can be bounded by
\begin{align*}
    C \int_0^t f(s) \norm{\delta u(s) }_2^2 \dd s,
\end{align*}
where the constant $C$ and the function $f$ can be controlled by norms of the strong solution $u_2$. Since we exclude the
vacuum formation in Section \ref{sec:vac}, for every $s>0$, we obtain
\begin{align*}
    \frac{1}{C} \norm{\delta u(s)}_2^2 \leq \norm{\sqrt{\rho_1(s)} \delta u(s)}_2^2,
\end{align*}
which allows us to conclude the weak-strong uniqueness using \textit{Grönwall's inequality}.

\subsection{Comparison with existing literature}\label{sec:compare}

Let us briefly compare our approach with existing literature. A widely used technique to show the uniqueness of solutions 
of \eqref{eq:ns} is the transition to Lagrangian coordinates, see 
\cite{AbidiGuiZhang2023,ChenZhangZhao2016,Danchin2023,DanchinMucha2012,PaicuZhangZhang2013}.  This gives 
a new system of equations which is equivalent to \eqref{eq:ns} under high regularity assumptions on the velocity fields 
$u_1$ and $u_2$, for instance, if $u_i$ satisfies
\begin{align*}
    \nabla u_i \in L^1_\loc(0,\infty;L^\infty(\R^d)) \quad \text{for} \quad i=1,2.
\end{align*}
Therefore, the approaches based on such a transformation can only provide uniqueness results for classes of solutions 
more regular than the one we deal with. Here, we rely on the so-called relative 
energy method to work in a low-regularity framework. This method allows us to show that: if the existence of strong solutions is 
known, then it is unique among a large class of weak solutions. 

As described above, the goal is to apply \textit{Grönwall's inequality} to the equation satisfied by the relative energy 
\eqref{eq:RE}. Similar approaches were already employed in \cite{Li2017} and \cite{DanchinWang2022}. Let us describe the 
main differences.

In \cite{Li2017}, the approach developed by Li cannot be used in our weak-strong setting as it is restricted to weakly 
differentiable initial densities. The reason for this 
is that the control of the product term in \eqref{eq:productintro} is done by deriving $L^{3/2}$ stability estimates for the transport
equation in \eqref{eq:nsdiffintro}, requiring a control on $\nabla \rho_0$. In \cite{DanchinWang2022}, Danchin and Wang derived an estimate for the transport equation in $\dot{H}^{-1}(\R^d)$. For 
$d=2$, this requires the control of
\begin{align}\label{eq:danchinwang2d}
    \int_0^t \left( \norm{t \dot{u}_2(t)}^2_\infty + \norm{t\nabla^2 \dot{u}_2(t)}^q_p \right) \dd t, \quad \frac{1}{p} +
    \frac{1}{q} = \frac{3}{2},
\end{align}
and, if $d=3$,
\begin{align}\label{eq:danchinwang3d}
    \int_0^t \left( \norm{t \dot{u}_2(t)}^2_\infty + \norm{t\nabla^2 \dot{u}_2(t)}_3 \right) \dd t,
\end{align}
quantities that are unlikely to be bounded for the strong solutions we are considering.

\subsection{Outline of the paper}
The rest of the paper is dedicated to the proof of Theorem \ref{thm:weakstrong1}. It is structured as follows: 
In Section 2, we show the absence of vacuum formation for every Leray-Hopf weak solution. Section 3 is dedicated to showing that the strong 
solutions (as defined in Definition \ref{def:strongsol}) are admissible test functions for our weak solutions. 
In Section 4, we employ the relative energy method to justify our weak-strong uniqueness result. Some technical lemmas are relegated to the appendix.

\section{No vacuum formation}\label{sec:vac}

In this section, we analyze the weak solutions of the continuity equation in dimensions $d=2,3$:
\begin{align}\label{eq:conti}
    \begin{cases}
        \partial_t \rho + \dive (\rho u) = 0, & t \in (0,T), \ x \in \R^d \\
        \rho(0,x) = \rho_0(x), & x \in \R^d,
    \end{cases}
\end{align} 
where $0< T \leq \infty$, $\rho \colon [0,T) \times \R^d \to \R $ is the unknown and $u \colon [0,T) \times 
\R^d \to \R^d $ is a given velocity field such that $\dive u= 0$. We analyze two different scenarios:

\begin{itemize}
    \item[(1)] The velocity field has the regularity of a \textit{Leray-Hopf} weak solution in the sense of Definition 
    \ref{def:lerayhopfsol}.
    \item[(2)] The velocity has the regularity of a strong solution, i.e. there exists an $\eta > d/2-1$ such that $u \in 
    \mathcal{S}^\eta(0,T)$.
\end{itemize}
In case (2), we benefit from numerous favorable properties as $u$ is regular, allowing us to apply standard theory. 
In case (1), we need a suitable adaptation of the \textit{DiPerna-Lions} theory. We will see that in both cases 
there exists a flow associated with the velocity field and that the formation of vacuum is excluded. This observation is crucial in our analysis to be able to employ the relative energy method.

\subsection{Flows and weak solutions} 

First, we assume that \(u \) is a strong solution of \eqref{eq:ns}. This implies that $u \in \mathcal{S}^\eta(0,T)$ 
and by Lemma \ref{lem:adddecayprop}, we have \(u \in L^1_\loc((0,T);W^{1,\infty}(\R^d))\). With this 
regularity property in hand, the existence and uniqueness of the flow associated to \(u\) is clear, namely the map 
\(X \colon [0,T) \times \R^d \to \R^d\), which solves, for every \(x \in \R^d\), the system of ordinary differential 
equations:
\begin{align}\label{eq:flowode}
    \begin{cases}
    \dot X(t,x) &= u(t,X(t,x)), \quad t \in (0,\infty), \\
    X(0,x) &= x.
    \end{cases}
\end{align} 

As a consequence of the Picard-Lindelöf/Cauchy-Lipschitz theorem from the theory of ordinary differential equations, we 
have the following properties.
\begin{lemma}\label{lem:propertiesflow}
    Let $u\in \mathcal{S}^\eta(0,T)$. Then the flow map $X \colon [0,T) \times \R^d \to \R^d$ is well-defined 
    and bi-Lipschitz, i.e. $X(t,\cdot) \colon \R^d \to \R^d$ is an invertible Lipschitz mapping for every 
    $t \in (0,T)$ and its inverse mapping $X^{-1}(t, \cdot) \colon \R^d \to \R^d$ is Lipschitz continuous as
    well. Moreover, for every $t \in [0,T)$, we have
    \begin{align}
        \exp \left( - \int_0^t \norm{\nabla u(s)}_\infty \dd s \right) \leq \norm{X(t)}_\lip
        \leq \exp \left( \int_0^t \norm{\nabla u(s)}_\infty \dd s \right),
    \end{align}
   and for every $t \in [0,T)$ and every $x \in \R^d$, we have
    \begin{align}\label{eq.incomcondi}
        J X(t,x) = \det D X(t,x) = 1.
    \end{align}
\end{lemma}

When dealing with weak solutions, we need to relax the assumptions on $u$, as the Leray-Hopf weak solutions do not 
satisfy a Lipschitz bound. From the result of DiPerna and Lions \cite[Theorem III.2]{DiPernaLions1989}, it is known that 
one can also show the existence and uniqueness of a flow under the following assumptions:
\begin{align}\label{eq:condisonu}
    u \in L^1(0,T;W^{1,1}_\loc(\R^d)), \quad \frac{u}{1+|x|} \in L^1(0,T;L^1(\R^d)+L^\infty(\R^d))\quad \text{and}\quad \dive u = 0.
\end{align}
In Lemma \ref{lem:decomptimedep}, we show that the Leray-Hopf weak solutions satisfy \eqref{eq:condisonu} so one can 
employ a characteristic formulation for the continuity equation to show that there is no formation of vacuum. This is 
presented in the next result.

\begin{proposition}
\label{prop:exunicontieq}
    Let $u\in L^1_\loc(0,T;W^{1,1}_\loc(\R^d))$ with $\dive u=0$ for almost every $t \in (0,T).$ If
    \begin{align}\label{eq:decay18}
        \frac{u}{1+|x|} \in L^1(0,T;L^1(\R^d)+L^\infty(\R^d)),
    \end{align}
    then \eqref{eq:conti} admits a unique weak solution $\rho \in L^\infty((0,T) \times \R^d)$ which, for almost every 
    $(t,x)\in(0,T)\times \R^d$, is given by
    \begin{align*}
        \rho(t,x) = \rho_0(X^{-1}(t,x)),
    \end{align*}
    where $X$ is the flow associated to $u$ defined in \eqref{eq:flowode}.
\end{proposition}

\begin{proof}
    The result is a direct consequence of \cite[Prop. II.1]{DiPernaLions1989}, \cite[Cor. II.1]{DiPernaLions1989} and 
    \cite[Thm. III.1]{DiPernaLions1989}.
\end{proof}

\begin{remark}\label{rmk:transruleinc}
    The incompressibility condition $\dive u=0$ implies that $X(t,\cdot)$ and $X^{-1}(t,\cdot)$ are measure-preserving 
    diffeomorphisms and, in particular, we have (\ref{eq.incomcondi}). As a consequence, for $\alpha\in\{-1,1\}$, one has
    \begin{align*}
        \int_0^T f(t,X^\alpha(t,x)) \dd x \dd t = \int_0^T f(t,x) \dd x \dd t
    \end{align*}
    for every $f \in L^1((0,T) \times \R^d).$
\end{remark}

\begin{remark}\label{rmk:com}
   A key idea of the theory developed in \cite{DiPernaLions1989} is to show
    that every weak solution $\rho$ of (\ref{eq:conti}) is renormalized, i.e. for every $\beta \in C^1(\R^d)$, it holds, in the sense of distributions, 
    \begin{align*}
        \partial_t \beta(\rho) + \dive (\beta(\rho) \cdot u) = 0.
    \end{align*}
This is done by developing commutator estimates, for instance, from 
    \cite[Lemma 2.3, p.43]{Lions1996}, for $v \in H^1(\R^d;\R^d)$ and $g\in L^\infty(\R^d)$, there exists a constant $C>0$ independent of $\varepsilon$ such that 
\begin{align}
    \norm{\dive(g v_\eps - \div (g_\eps v))}_{L^2(\R^d)} \leq C\norm{v}_{W^{1,2}(\R^d)} \norm{g}_{L^\infty(\R^d)},
\end{align}
and
\begin{align}\label{eq:commutatorestimate}
    \lim_{\eps \to 0}\norm{\dive((g v)_\eps - g_\eps v)}_{L^2(\R^d)} \to 0,
    \end{align}
    where $f_\varepsilon$ denotes the mollification in space of a time-space function $f$ in $L^1_{\loc}((0,T)\times\R^d)$.
\end{remark}

With these results in hand, if we show that a \textit{Leray-Hopf} weak solution satisfies \eqref{eq:decay18}, then the 
characteristic formulation allows us to exclude the formation of vacuum. This step is done in Lemma 
\ref{lem:decomptimedep} and leads to the main result of this section.
\begin{theorem}[No Vacuum Formation]\label{thm:novacuum}
    Let \((\rho_0,u_0) \in L^\infty(\R^d) \times L_\sigma^2(\R^d;\R^d)\) and assume that
    \[ 0 < c_0 \leq \rho_0 \leq C_0. \] Let \((\rho,u)\) be a Leray-Hopf weak solution of \eqref{eq:ns} in the 
    sense of Definition \eqref{def:lerayhopfsol} with the initial data $(\rho_0,u_0).$  Then, for almost every \((t,x) \in 
    [0,\infty) \times \R^d\), \[ 0 < c_0 \leq \rho(t,x) \leq C_0 \quad \text{and} \quad u \in 
    L^\infty((0,\infty); L^2(\R^d)). \]  
\end{theorem}

\begin{proof} 
First, we have that $\rho$ is the unique solution of \eqref{eq:conti} in $L^\infty((0,T) \times \R^d)$ if $u$ 
satisfies the conditions stated in \eqref{eq:condisonu}. Under this assumption, by Proposition \ref{prop:exunicontieq}, we obtain that 
$\rho$ satisfies
\begin{align}\label{eq:formunovacuum}
    \rho(t,\cdot)= \rho_0(X^{-1}(t,\cdot)), \quad t>0.
\end{align}
From \eqref{eq:formunovacuum}, it follows directly that there is no formation of vacuum. Concerning the conditions in \eqref{eq:condisonu},  
\begin{align*}
    u \in L^1_\loc(0,T;W^{1,1}_\loc(\R^d)) \quad \text{and}\quad \dive u =0
\end{align*}
follow from the definition of weak solutions and the condition 
\begin{align*}
    \frac{u}{1+|x|} \in L^1(0,T;L^1(\R^d)+L^\infty(\R^d))
\end{align*}
is satisfied thanks to Lemma \ref{lem:decomptimedep}.
\end{proof}

\section{Admissible test functions for the continuity and momentum equation}\label{sec:admin}

In this section, we show that the strong solutions are admissible test functions for the weak formulation. 
Throughout the section, let $q>1$ and $0<T\leq \infty.$ We define the spaces
\begin{align*}
    X_{2,q,T} & := W^{1,q}(0,T;L^2_\sigma(\R^2)) \cap L^q(0,T;H^2(\R^2)),\\
    X_{3,q,T} & := W^{1,q}(0,T;L^2_\sigma(\R^3)) \cap L^q(0,T;H^2(\R^3)) \cap L^4(0,T;L^6(\R^3)).
\end{align*}
In both cases $d=2,3$,
\begin{align*}
    X_{d,q,T} \hookrightarrow C([0,T);L^2_\sigma(\R^d)) \quad \textrm{ and } \quad 
    X_{d,q,T} \hookrightarrow L^2(0,T;\dot{H}^1(\R^d)).
\end{align*}
Furthermore, we define the space
\begin{align*}
    \mathcal{B}_{q,T} &= \left\{ \varphi \in X_{d,q,T}: \varphi \in 
    H^1(\eps,\eps^{-1} \wedge T;H^1(\R^d)),\ \eps>0 \right\}.
\end{align*}
Let $0<t_1<t_2<T$ and let $\varphi \in \mathcal{B}_{q,T}$. From the paper of Masuda \cite{Masuda1984}, we know that 
there exists a sequence $\{\varphi_n\}_{n \in \N} \subset C^\infty([t_1,t_2];C^\infty_{c,\sigma}(\R^d))$ such that
\begin{align*}
    \norm{\varphi_n - \varphi}_{H^1(t_1,t_2;H^1(\R^d))} \underset{n\to\infty}{\to} 0,
\end{align*}
and, in particular, for every $\tau \in [t_1,t_2]$, we have
\begin{align*}
    \norm{\varphi_n(\tau) - \varphi(\tau)}_{H^1(\R^d)} \underset{n\to\infty}{\to} 0.
\end{align*}
The following lemma shows the convergence of the approximations of nonlinear terms.

\begin{lemma}\label{lem:convproducts}
    Let $\varphi \in \mathcal{B}_{q,T}$, $0<t_1<t_2<T$ and $\{\varphi_n\}_{n \in \N}$ be as above. Suppose that
    \begin{align*}
        u \in L^\infty(0,T;L^2_\sigma(\R^d)) \quad \text{and} \quad \nabla u \in L^2(0,T;L^2(\R^d)).
    \end{align*}
   We have 
    \begin{align*}
        \partial_t \varphi \cdot \varphi, \,\partial_t \varphi \cdot u, \,u \otimes \varphi : \nabla \varphi, \,
        u \otimes u : \nabla \varphi \in L^1((0,t_2) \times \R^d),
    \end{align*}
    and the following convergence results hold in $L^1((t_1,t_2) \times \R^d)$:
    \begin{align*}
        \partial_t \varphi_n \cdot \varphi_n \underset{n\to\infty}{\to} \partial_t \varphi \cdot \varphi, \quad
        \partial_t \varphi_n \cdot u \underset{n\to\infty}{\to}\partial_t \varphi \cdot u,\\
        u \otimes \varphi_n : \nabla \varphi_n \underset{n\to\infty}{\to} u \otimes \varphi : \nabla \varphi, \quad
        u \otimes u : \nabla \varphi_n \underset{n\to\infty}{\to} u \otimes u : \nabla \varphi.
    \end{align*}
\end{lemma}

Next, we investigate the weak continuity of $\rho u$ when $(\rho,u)$ is a \textit{Leray-Hopf} weak solution. 
\begin{lemma}\label{lem:Contintimemom}
    Let $(\rho_0,u_0) \in L^\infty(\R^d) \times L^2_\sigma(\R^d)$ and $(\rho,u)$ be a 
    \textit{Leray-Hopf} weak solution with respect to $(\rho_0,u_0)$. There 
    exists a set $I_w \subset [0,\infty)$ with $\lambda([0,\infty) \setminus I_w)=0$ and $0 
    \in I_w$ such that
    \begin{itemize}
        \item[(i)] For every $\varphi \in H_\sigma^1(\R^d;\R^d)$ and $t,s \in I_w$, we have
        \begin{align}\label{eq:testmontimeinde}
        \begin{aligned}
            \int_{\R^d} \rho(t)u(t) \cdot \varphi\dd x &- \int_{\R^d} 
            \rho(s)u(s) \cdot \varphi\dd x\\ &= \int_s^t\int_{\R^d}  \rho u 
            \otimes u : \nabla \varphi \dd x \dd \tau 
            - \nu \int_s^t\int_{\R^d} \nabla u: \nabla \varphi \dd x \dd \tau.
        \end{aligned}
         \end{align}
        \item[(ii)] $\rho u \in C_w(I_w;L^2_\sigma(\R^d)).$
    \end{itemize}
\end{lemma}

\begin{proof}
    Let $(\varphi_n)_{n \in \N} \subset C^\infty_{0,\sigma}(\R^d)$ be a dense sequence in $H^1_\sigma
    (\R^d)$ and $I_n$ be the set of Lebesgue points of the mapping 
    \begin{align*}
        t \mapsto \int_{\R^d} \rho(t) u(t) \varphi_n \dd x.
    \end{align*}
    By Lebesgue's differentiation theorem, we have $\lambda([0,\infty) \setminus I_n)=0$ for every $n \in \N$ where 
    $\lambda$ is the Lebesgue measure. Defining
    \begin{align*}
        I_w = \bigcap_{n \in \N} I_n \cup \{0\},
    \end{align*}
    one has $\lambda([0,\infty) \setminus I_w)=0.$ In the next step, we justify (\ref{eq:testmontimeinde}) for every 
    $n \in \N$. Note that it suffices to show is it true for every $t \in I_{\omega}$, with $t>0$ and $s=0.$ The general case 
    follows by subtraction. Let $\psi \colon \R \to \R$ be a smooth, decreasing function and assume that $\psi \equiv 1$ 
    on $(-\infty,0)$ and $\psi \equiv 0$ on $(1,\infty)$. For every $\eps > 0$ and every $n \in \N$, set
    \begin{align*}
        \varphi_{\eps,n}(\tau,x) = \psi \left(\frac{t-\tau}{\eps} \right) \varphi_n(x).
    \end{align*}
     Then $\varphi_{\eps,n}$ is an admissible 
    test function for $(\rho,u)$. Using Lebesgue's differentiation theorem and the 
    dominated convergence theorem one concludes that, for every $n\in \mathbb{N}$, \begin{align*}
        \int_{\R^d} \rho(t)u(t) \cdot \varphi_n \dd x &- \int_{\R^d} 
        \rho_0 u_0 \cdot \varphi_n \dd x\\ &= \int_0^t \int_{\R^d}  \rho u 
        \otimes u : \nabla \varphi_n \dd x \dd \tau 
        - \nu \int_0^t \int_{\R^d} \nabla u \nabla \varphi_n \dd x \dd \tau.
    \end{align*}
    The equality in (\ref{eq:testmontimeinde}), for general $\varphi 
    \in H_\sigma^1(\R^d;\R^d)$, follows by passing to the limit $n\to\infty$ and the fact that $\rho u \in    
    L^\infty(0,\infty;L^2(\R^d)).$ This proves (i). 
    
    In order to show (ii), recall that, for every $\varphi \in H^1(\R^d;\R^d)$, we have
    \begin{align*}
        \rho u \otimes u : \nabla \varphi,\,  \nabla u : \nabla \varphi \in L^1_\loc((0,T) \times \R^d).
    \end{align*}
     From \eqref{eq:testmontimeinde}, we deduce that, for every $t \in I_w$ and every $\varphi \in H^1_\sigma(\R^d;\R^d)$,
    \begin{align}\label{eq:weakconH1}
        \lim_{I_w \ni s \to t} \int_{\R^d} \rho(s)u(s) \cdot \varphi \dd x =
        \int_{\R^d} \rho(t)u(t) \cdot \varphi \dd x.
    \end{align}
    Since $\rho u \in L^\infty(0,\infty;L^2_\sigma(\R^d))$ and 
    $H^1_\sigma(\R^d)$ is dense in $L^2_\sigma(\R^d)$, we deduce that \eqref{eq:weakconH1} is true for every $\varphi \in
    L^2_\sigma(\R^d)$ which concludes the proof.
\end{proof}

The next lemma shows that functions in $\mathcal{B}_{q,T}$ can be used as test functions for (\ref{eq:ns}).

\begin{lemma}\label{lem:testins}
    Let $(\rho,u)$ be a weak Leray-Hopf solution of (\ref{eq:ns}) with respect to initial $(\rho_0,u_0) \in L^\infty(\R^d)
    \times L^2_\sigma(\R^d))$ and let $\varphi \in \mathcal{B}_{q,T}.$ For every $t \in (0,T)$, we have
    \begin{align}\label{eq:testcontBpt}
        \frac{1}{2} \int_{\R^d} \rho(t) |\varphi(t)|^2 \dd x - \frac{1}{2} \int_{\R^d} \rho(0) |\varphi(0)|^2 \dd x
        = \int_0^t \int_{\R^d} \rho \partial_t \varphi \cdot \varphi + \rho u \otimes \varphi : 
        \nabla \varphi \dd x \dd \tau,
    \end{align}
    and, for almost every $t \in (0,T)$,
    \begin{align}\label{eq:testmomBpt}
    \begin{aligned}
        - \int_0^t \int_{\R^d} & \rho \partial_t \varphi \cdot u+ 
        \rho u \otimes u : \nabla \varphi \dd x \dd s  + \int_{\R^d}  \rho(t)u(t) 
        \cdot \varphi(t) \dd x \\
        = & - \nu \int_0^t \int_{\R^d} \skal{\nabla u,\nabla \varphi} \dd s 
        + \int_{\R^d} \rho_0 u_0 \cdot \varphi(0) \dd x.
    \end{aligned}
    \end{align}
\end{lemma}

\begin{proof}
    We start with the verification of (\ref{eq:testcontBpt}).
    Fix $0<s<t<\infty$ and let $\psi \in C^\infty((0,\infty))$ be a cut-off function with $\psi \equiv 1$ on $[s,t]$
    and $\psi \equiv 0$ on $(0,s/2) \cup (2t,\infty).$ Clearly, $\psi \cdot \varphi \in H^1(0,2t;H^1(\R^d))$ and we 
    can find a sequence $\{\varphi_n\} \subset C^\infty([0,2t];C^\infty(\R^d))$ such that $\varphi_n \to \psi 
    \varphi$ in $H^1(0,2t;H^1(\R^d))$. In particular, $\varphi_n \to \varphi$ in $H^1(s,t;H^1(\R^d))$ and due to
    the embedding $H^1(s,t;H^1(\R^d)) \hookrightarrow C^{1/2}([s,t];H^1(\R^d))$, for every $\tau \in [s,t]$, we have 
    \begin{align}\label{eq:ptwsconvergence}
        \varphi_n(\tau) \underset{n\to\infty}{\to} \varphi(\tau) \quad \textrm{ in } L^2(\R^d).
    \end{align}
    Since $\rho \in 
    C_{w^*}([0,\infty);L^\infty(\R^d))$, Lebesgue's differentiation theorem  implies that, 
    for every $n \in \N$, 
\begin{align}\label{eq:formulaptwsn}
        \frac{1}{2} \int_{\R^d} \rho(t) |\varphi_n(t)|^2 \dd x - \frac{1}{2} \int_{\R^d} \rho(s) |\varphi_n(s)|^2 
        \dd x = \int_s^t\int_{\R^d} \rho \partial_t \varphi_n \cdot \varphi_n + \rho u \otimes \varphi_n : 
        \nabla \varphi_n \dd x \dd \tau.
    \end{align} 
    Note that we used the identity
    \begin{align*}
        \frac{1}{2} u \cdot \nabla |\varphi_n|^2 = u \otimes \varphi_n : \nabla \varphi_n.
    \end{align*}
    Thanks to \eqref{eq:ptwsconvergence}, the terms on the left-hand side of \eqref{eq:formulaptwsn} converge to
    \begin{align*}
        \frac{1}{2} \int_{\R^d} \rho(t) |\varphi(t)|^2 \dd x - \frac{1}{2} \int_{\R^d} \rho(s) |\varphi(s)|^2 \dd x.
    \end{align*}
    With Lemma \ref{lem:convproducts}, we have
    \begin{align*}
        \int_s^t\int_{\R^d} \rho \partial_t \varphi_n \cdot \varphi_n \dd x
        + \rho u \otimes \varphi_n : \nabla \varphi_n \dd x \dd \tau 
        \underset{n\to\infty}{\to} 
        \int_s^t\int_{\R^d} \rho \partial_t \varphi \cdot \varphi 
        + \rho u \otimes \varphi : \nabla \varphi \dd x \dd \tau.
    \end{align*}
    This gives (\ref{eq:testcontBpt}) for every $t,s>0.$  
    
    It remains to prove the statement for $s=0.$ Observe that 
    \begin{align*}
        \lim_{s\to 0}\frac{1}{2} \int_{\R^d} \rho(s) |\varphi(s)|^2 \dd x=
        \frac{1}{2} \int_{\R^d} \rho_0 |\varphi(0)|^2 \dd x \quad 
    \end{align*}
    since
    \begin{equation*}
       W^{1,p}(0,\infty;L^2(\R^d)) \hookrightarrow C([0,T);L^2(\R^d)) \quad \textrm{ and } \quad
       \rho \in C_{w^*}([0,T);L^\infty(\R^d)).
    \end{equation*}
    The claim follows now from the fact that $\partial_t \varphi \cdot \varphi, \rho u \otimes \varphi : \nabla \varphi
    \in L^1((0,t) \times \R^d)$.

    We briefly discuss the proof of (\ref{eq:testmomBpt}). For every $\varphi \in \mathcal{B}_{q,T}$, we define
    \begin{align*}
        \Phi_\varphi \colon [0,T) \to \R, \quad t \mapsto \int_{\R^d} \rho(t) u(t) \cdot \varphi(t) \dd x.
    \end{align*}
    Since $\varphi \in C([0,T);L^2_\sigma(\R^d)),$ Lemma \ref{lem:Contintimemom} implies that $\Phi_\varphi|_{I_w}$
    is continuous. Hence, for every $t \in I_w \setminus \{0\}$
    \begin{align}\label{eq:lebesguephi}
        \lim_{\eps \to 0} \frac{1}{2 \eps} \int_{t- \eps}^{t+\eps} \Phi_\varphi(\tau) \dd \tau = \Phi_\varphi(t).
    \end{align}
    Let $\varphi \in \mathcal{B}_{q,T}$, $t,s \in I_w \setminus \{0\}$ and $\psi,\{\varphi_n\}_{n \in \N}$
    be as above. As for the continuity equation, using \eqref{eq:lebesguephi}, we have
    \begin{align}\label{eq:testmomphin}
    \begin{aligned}
        - \int_0^t \int_{\R^d} & \rho_1 \partial_t \varphi_n \cdot u+ 
        \rho u \otimes u : \nabla \varphi_n \dd x \dd s  + \int_{\R^d}  \rho(t)u(t) 
        \cdot \varphi_n(t) \dd x \\
        = & - \nu \int_0^t \int_{\R^d} \nabla u : \nabla \varphi_n \dd s 
        + \int_{\R^d} \rho(s) u(s) \cdot \varphi_n(s) \dd x.
    \end{aligned}
    \end{align}
    With Lemma \ref{lem:convproducts}, we see that (\ref{eq:testmomphin}) remains true when $n \to \infty$ for 
    every $t,s \in I_w \setminus \{0\}.$ Let $(s_n)_{n \in \N} \subset I_w \setminus \{0\}$ with $s_n 
    \underset{n \to \infty}{\to} 0.$ Note that \eqref{eq:testmomphin} is true with $s_n$ and $\varphi.$
    Hence, using that $\Phi_\varphi(s_n) \underset{n \to \infty}{\to} \Phi_\varphi(0)$ and Lemma \ref{lem:convproducts} 
    completes the proof.
\end{proof} 

Now that we have shown that $\mathcal{B}_{q,T}$ is a good admissible set, we show that the strong solutions we consider 
belong to this set.
\begin{lemma}\label{Lem:AdminTest}
    Let $(\rho_0,u_0) \in L^\infty(\R^d) \times L^2_\sigma(\R^d)$ and let $T>0.$ Suppose that $(\rho_1,u_1)$ is a weak
    Leray-Hopf solution with respect to the initial data $(\rho_0,u_0)$ and that $(\rho_2,u_2)$ is a strong 
    solution of \eqref{eq:ns} on $(0,T)$ with respect to the same initial data $(\rho_0,u_0).$ Then there exists a
    $q>1$ such that $u_2\in\mathcal{B}_{q,T}$ and \eqref{eq:testcontBpt}-
    \eqref{eq:testmomBpt} hold for $\varphi=u_2$ and $(\rho,u)=(\rho_1,u_1)$.
\end{lemma}
\begin{proof}
The proof follows from employing Lemma \ref{lem:testins} thanks to the extra regularity properties on the strong solution that we derive 
in Lemma \ref{lem:adddecayprop}.
\end{proof}

\section{Weak-strong uniqueness result}\label{sec:wsu}
In this section, we prove Theorem \ref{thm:weakstrong1}.  Let $(\rho_1,u_1)$ be a 
\textit{Leray-Hopf} weak solution with respect to the initial data $(\rho_0,u_0)$, and let
$(\rho_2,u_2)$ be a strong solution of (\ref{eq:ns}) as in Definition \ref{def:strongsol} 
with respect to the same initial data $(\rho_0,u_0)$ with maximal existence interval $(0,T^*)$.

\subsection{Relative energy method}
The main goal of this section is the verification of \eqref{eq:RE}, more precisely, we prove Lemma \ref{lem:badgronwall} below. To this end, we introduce the notation 
\begin{align*}
    \delta \rho = \rho_1-\rho_2, \quad \delta u = u_1 - u_2 \quad
    \textrm{ and } \quad \dot{u}_2 = \partial_t u_2 + (u_2 \cdot \nabla)u_2.
\end{align*}
In this section, we fix $(\rho_1,u_1)$ and $(\rho_2,u_2)$ such that $\rho_1,\rho_2 \in C_{w^*}([0,T^*);L^\infty(\R^d))$ 
and $u_2 \in C([0,T^*);L^2(\R^d))$. 

\begin{lemma}\label{lem:badgronwall}
We have \begin{align}\label{eq:badgronwalleq}
        \frac{1}{2}\norm{\sqrt{\rho_1(t)} \delta u(t)}_2^2 + \nu \int_0^t 
        \norm{\nabla \delta u}^2_2 \dd s \leq -\int_0^t \int_{\R^d} \delta \rho 
        \dot{ u_2} \cdot \delta u +\rho_1 \delta u \otimes \delta u : \nabla u_2 
        \dd x \dd s
    \end{align}
    for almost every $t \in (0,T^*).$
\end{lemma}

\begin{proof}
   For $t>0$, consider the functional
    \begin{align*}
        E_{rel}(t):= \frac{1}{2} \int_{\R^d} \rho_1(t) \abs{u_1(t)-u_2(t)}^2 \dd x 
        = \frac{1}{2} \norm{\sqrt{\rho_1(t)} \delta u(t)}^2_2.
    \end{align*}
  Using the energy inequality \eqref{eq:EE}, for every $t \in (0,T^*)$, we have
    \begin{align}\label{eq:Erel}
         E_{rel}(t) = & \frac{1}{2}\norm{\sqrt{\rho_1(t)}u_1(t)}_2^2
        -\skal{\rho_1(t)u_1(t),u_2(t)}_2 +\frac{1}{2}\norm{\sqrt{\rho_1(t)}u_2(t)}_2^2\nonumber\\
        = &\frac{1}{2}\norm{\sqrt{\rho_1(t)}u_1(t)}_2^2 -\skal{\rho_1(t)u_1(t),u_2(t)}_2
        +\frac{1}{2}\norm{\sqrt{\rho_2(t)}u_2(t)}_2^2 + \frac{1}{2} 
        \skal{\delta \rho(t)u_2(t),u_2(t)}\\
        \leq & \norm{\sqrt{\rho_0}u_0}_2^2
        - \nu \int_0^t \left( \norm{\nabla u_1}^2_2 + \norm{\nabla u_2}^2_2 \right) 
        \dd s -\skal{\rho_1(t)u_1(t),u_2(t)}_2 
        + \frac{1}{2}\skal{\delta \rho(t)u_2(t),u_2(t)}_2.\nonumber
    \end{align}
    Thanks to Lemma \ref{Lem:AdminTest}, we infer that, for every $t \in (0,T^*)$ and $i=1,2$, 
    \begin{align*}
        \frac{1}{2}\int_{\R^d} \rho_i(t)|u_2(t)|^2 \dd x - 
        \frac{1}{2}\int_{\R^d} \rho_0|u_2(0)|^2 \dd x 
        = \int_0^t \int_{\R^d} \rho_i \partial_t u_2 \cdot u_2 \dd x \dd s
        + \int_0^t \int_{\R^d} \rho_i u_i \otimes u_2 : \nabla u_2 \dd x \dd s. 
    \end{align*}
    Subtraction yields
    \begin{align}\label{eq:12}
        \frac{1}{2}\int_{\R^d} \delta \rho(t)|u_2(t)|^2 \dd x = \int_0^t \int_{\R^d}  \delta \rho 
        \partial_t u_2 \cdot u_2 + \delta \rho u_2 \otimes u_2 : \nabla u_2 
        + \rho_1 \delta u \otimes u_2 : \nabla u_2  \dd x \dd s.
    \end{align} Furthermore, Lemma \ref{Lem:AdminTest} implies that, 
    for almost every $t \in (0,T^*)$,
    \begin{align}\label{eq:13}
        - \int_0^t \int_{\R^d} & \rho_1 \partial_t u_2 \cdot u_1
        + \rho_1 u_1 \otimes u_1 : \nabla u_2 \dd x 
        \dd s  + \int_{\R^d}  \rho_1(t)u_1(t) \cdot u_2(t) \dd x \\
        = & - \nu \int_0^t \int_{\R^d} \nabla u_1 :\nabla u_2 \dd x \dd s 
        + \int_{\R^d} \rho_0 |u_0|^2 \dd x.\nonumber
    \end{align}Using \eqref{eq:12}-\eqref{eq:13} in \eqref{eq:Erel}, for almost every 
    $t \in (0,T^*),$ we have
    \begin{align*}
        E_{rel}(t) \leq & \norm{\sqrt{\rho_0}u_0}_2^2
        - \nu \int_0^t \left( \norm{\nabla u_1}^2_2 + \nu \norm{\nabla u_2}^2_2 \right) 
        \dd s + \nu \int_0^t \nabla u_1 : \nabla u_2 \dd x \dd s\\
        & - \norm{\sqrt{\rho_0}u_0}_2^2 - \int_0^t \int_{\R^d} \rho_1 \partial_t u_2 
        \cdot u_1+ \rho_1 u_1 \otimes u_1 : \nabla u_2 \dd x \dd s \\
        &+ \int_0^t \int_{\R^d} \delta \rho \partial_t u_2 \cdot u_2 + \delta \rho u_2 
        \otimes u_2 : \nabla u_2 + \rho_1 \delta u \otimes u_2 : \nabla u_2  \dd x \dd s.
    \end{align*}
   Using the identity
    \begin{align*}
        (v \cdot \nabla) v \cdot u = v \otimes u : \nabla v, \quad u,v \in H^1(\R^d),
    \end{align*}
     multiplying the momentum equation 
    of $u_2$ by $u_1$ and integrating by parts, for every $t \in (0,T^*)$, we obtain
    \begin{align*}
        \int_0^t \int_{\R^d} \rho_2 \partial_t u_2 \cdot 
        u_1 + \rho_2 u_2 \otimes u_1 : \nabla u_2 \dd x \dd s 
        = - \nu \int_0^t \int_{\R^d} \nabla u_1 : \nabla u_2 \dd x \dd s.
    \end{align*}
Hence, for almost every $t \in (0,T^*)$,
    \begin{align*}
        E_{rel}(t) \leq & - \nu \int_0^t \norm{\nabla \delta u}^2_2 \dd s + 
        \int_0^t \int_{\R^d} \rho_2 \partial_t u_2 \cdot u_1 + \rho_2 u_2 \otimes u_1 
        : \nabla u_2 \dd x \dd s \\
        &- \int_0^t \int_{\R^d} \rho_1 \partial_t u_2 \cdot u_1+ \rho_1 u_1 \otimes u_1 
        : \nabla u_2 \dd x \dd s \\
        &+ \int_0^t \int_{\R^d} \delta \rho \partial_t u_2 \cdot u_2 + 
        \delta \rho u_2 \otimes u_2 : \nabla u_2 
        + \rho_1 \delta u \otimes u_2 : \nabla u_2  \dd x \dd s.
    \end{align*}
The identities
    \begin{align*}
        - \delta \rho \partial_t u_2 \cdot \delta u =
        \rho_2 \partial_t u_2 \cdot u_1 - \rho_1 \partial_t u_2 \cdot u_1 
        +\delta \rho \partial_t u_2 \cdot u_2
    \end{align*}
    and 
    \begin{align*}
        - \delta \rho (u_2 \cdot \nabla) u_2 \cdot \delta u 
        - \rho_1 \delta u \otimes \delta u : \nabla u_2 
        = & - \delta \rho u_2 \otimes \delta u : \nabla u_2  - \rho_1 \delta u \otimes 
        \delta u : \nabla u_2 \\ 
        = & \:\:\rho_2 u_2 \otimes u_1 : \nabla u_2  - \rho_1 u_1 \otimes u_1 : \nabla u_2\\
        & +\delta \rho u_2 \otimes u_2 : \nabla u_2 + \rho_1 \delta u \otimes 
        u_2 : \nabla u_2,
    \end{align*}
    yield, for almost every $t \in (0,T^*)$,
    \begin{align*}
        E_{rel}(t) \leq - \nu \int_0^t \norm{\nabla \delta u}^2_2 \dd s 
        -\int_0^t \int_{\R^d} \delta \rho \dot{ u_2} \cdot \delta u +\rho_1 \delta u \otimes 
        \delta u : \nabla u_2 \dd x \dd s
    \end{align*}
    
\end{proof}

\subsection{Stability estimate}
We now bound the r.h.s. of \eqref{lem:badgronwall}. One key point is the following $W^{-1,p}$-stability estimate.

\begin{proposition}\label{lem:wminus4stab}
Let $T \in (0,T^*).$ For every $\beta >0$, there exists a constant $C=C(\beta,T)>0$  such that, for every $t \in (0,T)$,
    \begin{align}
        \abs{\int_0^t \int_{\R^d} \delta \rho \dot{u}_2 \cdot \delta u \dd x
        \dd s} \leq C \int_0^t \sigma(s)^{1/2+\beta} \left( \int_0^s \tau^{-2\beta} \norm{\rho_1
        \delta u(\tau)}_p^2 \dd \tau \right)^{1/2} 
        \norm{\nabla (\dot{u}_2 \cdot \delta u)}_q \dd s,
    \end{align}
where $p=4$ if $d=2$ and $p=3$ if $d=3$, and $q$ denotes the conjugate index of $p$.
\end{proposition}

\begin{proof}
Fix $T \in (0,T^*).$ We divide the proof in several steps.

\textbf{Step 1.} (Regularization) For a fixed $\eps > 0$ and $i=1,2$, we have
\begin{align*}
    \partial_t (\rho_i)_\eps + \dive((\rho_i)_\eps u_i) = 
    \dive((\rho_i)_\eps  u_i-(\rho_i  u_i)_\eps)=:R_{i,\eps}.
\end{align*}
Here and in the following $f_\eps$ denotes the mollification in space of a function $f \in L^1_\loc((0,T)
\times \R^d).$ Subtraction yields
\begin{align}\label{eq:trasnportdeltarho}
    \partial_t (\delta \rho)_\eps + \dive((\delta \rho)_\eps u_2)=
    - \dive((\rho_1)_\eps \delta u ) + \delta R_\eps,
\end{align}
where $\delta R_\eps = R_{1,\eps}-R_{2,\eps}$. Recall that there exists a $C>0$ such that, almost everywhere,
\begin{align*}
    -C \leq \delta \rho(t,x) \leq C,
\end{align*}
 since $\delta \rho \in C_{w^*}([0,T];L^\infty(\R^d))$ and that for every $t \in (0,T)$ and almost everywhere on $\R^d$, we have
\begin{align}\label{eq:pointwise}
    (\delta \rho(t))_\eps \to \delta \rho(t).
\end{align}

\textbf{Step 2.} (Truncation at zero)
We introduce a family of smooth, non-decreasing cut-off functions  $\varphi_N \colon [0,\infty) \to \R$ parametrized by 
$N \in \N$ such that $\varphi_N=0$ for $0 \leq t \leq 1/N^2$ and $\varphi_N=1$ for $t \geq 1/N.$ Since $\dot{u}_2 \cdot 
\delta u \in L^1((0,T) \times \R^d)$ and  
\begin{align*}
    |(\delta \rho(t))_\eps (\varphi_N \dot{u}_2)\cdot \delta u| \leq C |\dot{u}_2 \cdot \delta u|,
\end{align*}
the dominated convergence theorem yields
\begin{align*}
    \int_0^t \int_{\R^d} \delta \rho \dot{u}_2 \cdot \delta u \dd x \dd s
    = \lim_{N \to \infty} \lim_{\eps \to 0} \int_0^t \int_{\R^d} (\delta 
    \rho)_\eps (\varphi_N \dot{u}_2) \cdot \delta u \dd x \dd s.
\end{align*}
Since $(\delta \rho)_\varepsilon$ solves the transport equation \eqref{eq:trasnportdeltarho}, denoting with $X \colon [0,T^*) \times \R^d \to \R^d$ the flow associated to $u_2$, using \cite[Proposition 6]{AmbrosioCrippa2014}, for every $x \in \R^d$ and every $s \in (0,T)$, we have
\begin{align*}
    (\delta \rho)_\eps (s,X(s,x)) = \int_0^s [-\dive((\rho_1)_\eps
    \delta u) + \delta R_\eps]|_{(\tau,X(\tau,x))} \dd \tau.
\end{align*}
With a change of variables, for every \(t \in (0,T)\), we obtain 
\begin{align}\label{eq:domconv}
     \int_0^t \int_{\R^d} & \delta \rho \dot{u}_2 \cdot \delta u \, dx \, ds \nonumber \\
    & = \lim_{N \to \infty} \lim_{\eps \to 0}  \int_0^t \int_{\R^d} \int_0^s
    (-\operatorname{div}((\rho_1)_\eps\delta u) + \delta R_\eps)|_{(\tau,X(\tau,x))} 
    (\varphi_N \dot{u}_2 \cdot \delta u)
    |_{(s,X(s,x))} \, d\tau \, dx \, ds 
\end{align}

\textbf{Step 3.} (Fubini's theorem)
In the next step, we exchange the order of integration. To this aim, it is 
necessary to show that the mapping defined by
\begin{align*}
    (s,\tau,x) \mapsto
    (-\dive((\rho_1)_\eps\delta u) + \delta R_\eps)|_{(\tau,X(\tau,x))} 
    (\varphi_N \dot{u}_2 \cdot \delta u)|_{(s,X(s,x))} =: I_1(\tau,x)I_2(s,x)
\end{align*}
is in $L^1((0,T) \times (0,T) \times \R^d).$ Using \textit{Hölder's inequality} in time and space, we get
\begin{align*}
     \int_0^T \int_0^t \int_{\R^d} \abs{I_1(\tau,x) I_2(s,x)} \dd x \dd \tau \dd s 
     & \leq \int_0^T \int_0^T \norm{I_1(\tau)}_2 \norm{I_2(s)}_2 \dd\tau \dd s\\
     & \leq T \left(\int_0^T \norm{I_1(\tau)}_2 \dd \tau \right)^\frac{1}{2}
     \left(\int_0^T \norm{I_2(s)}_2 \dd s \right)^\frac{1}{2}
\end{align*}
With the transformation rule from Remark \ref{rmk:transruleinc}, Remark \ref{rmk:com} \textit{Hölder's inequality} in space and time, and the 
inequalities
\begin{align*}
    |\dive((\rho_1(\tau))_\eps\delta u(\tau))| \leq \frac{C}{\eps}|\delta u(\tau)|
    \quad \textrm{ and } \quad
    \norm{R_{i,\eps}(\tau)}_2 \leq C \norm{\rho_i(\tau)}_\infty \norm{u_i(\tau)}_{H^1},
\end{align*}
we estimate $I_1$ as
\begin{align*}
    \int_0^T \norm{I_1(\tau)}_2 \dd \tau &\leq 2 \int_0^T \left( \frac{C^2}{\eps^2}\norm{\delta u(\tau)}_2^2 
    + \norm{\delta R_\eps(\tau)}_2^2\right) \dd \tau\\
    &\leq 2 \int_0^T \left( \frac{C^2}{\eps^2}\norm{\delta u(\tau)}_2^2 + \sum_{i=1,2}
    \norm{\rho_i(\tau)}_\infty^2 \norm{u_i(\tau)}_{H^1}^2 \right) \dd \tau.
\end{align*}
Since $\delta u \in L^2(0,T;H^1(\R^d))$, by Theorem \ref{thm:novacuum0}, we conclude
\begin{align*}
    \int_0^T \norm{I_1(\tau)}_2 \dd \tau < \infty.
\end{align*}
For the second term, we work differently for $d=2$ and $d=3$. In the case $d=2$, employing \textit{Hölder's inequality}, \textit{Gagliardo-Nirenberg inequality} and the definition of $\varphi_N$, we obtain
\begin{align*}
    \int_0^T \left(\int_{\R^2} |I_2(s,x)|^2 \dd x\right)^\frac{1}{2} \dd s 
    & = \int_{N^{-2}}^T \left(\int_{\R^d} \abs{\varphi_N \dot{u}_2\cdot \delta u}^2 \dd x \right)^{1\slash2}\dd s\\
    & \leq \int_{N^{-2}}^T \norm{\dot{u}_2}_4\norm{\delta u}_4\dd s\\
    & \leq  \int_{N^{-2}}^T \norm{\dot{u}_2}^{1\slash2}_2 \norm{\nabla \dot{u}_2}^{1\slash2}_2
    \norm{\delta u}^{1\slash2}_2 \norm{\nabla \delta u}^{1\slash2}_2 \dd s\\
    & \leq T^{1\slash2} \left(\int_{N^{-2}}^T \norm{\dot{u}_2}_2 \norm{\nabla \dot{u}_2}_2
    \norm{\delta u}_2 \norm{\nabla \delta u}_2 \dd s\right)^{1\slash2}\\
    & \leq T^{1\slash2}N^{2-\eta} \left(\int_{N^{-2}}^T \sigma(s)^{2-\eta} \norm{\dot{u}_2}_2 
    \norm{\nabla \dot{u}_2}_2 \norm{\delta u}_2 \norm{\nabla \delta u}_2 \dd s\right)^{1\slash2}\\
    & \leq T^{1\slash2}N^{2-\eta} \left( \int_0^T \sigma(s)^{4-2\eta} \norm{\dot u_2(s)}_2^2 
    \norm{\nabla \dot u_2(s)}_2^2   \dd s\right)^\frac{1}{4}\\
    &\quad \times \left( \int_0^T  \norm{ \delta u(s)}_2^2 \norm{\nabla \delta u(s)}_2^2 \dd s\right)^\frac{1}{4}\\
    & \leq T^{1\slash2}N^{2-\eta} \left(\sup_{s \in [0,t]} \sigma(s)^{2-\eta} \norm{\dot u_2(s)}_2^2\right)^\frac{1}{4} \left(\int_0^T \sigma(s)^{2-\eta} \norm{\nabla \dot u_2(s)}_2^2   \dd s\right)^\frac{1}{4}\\
    & \quad \times\left(  \norm{ \delta u}_{\infty,2} \norm{\nabla \delta u }_{2,2}\right)^{1\slash2}
\end{align*}
In the fourth step of the previous computation, we can introduce an arbitrary powers of $\sigma(s)$ under the integral 
since we are integrating away from zero. In the three-dimensional case, we omit the precise computation 
of the time weight. Using the interpolation inequality
\begin{align*}
    \norm{f}_4 \leq \norm{f}_2^{1/4} \norm{f}_6^{3/4} \leq \norm{f}_2^{1/4} \norm{\nabla f}_2^{3/4}, \quad
    f \in H^1(\R^3),
\end{align*}
we have
\begin{align*}
    \int_0^T \norm{I_2(s)}_2^2 \dd s 
    & \leq \int_{N^{-2}}^T \norm{\dot{u}_2}_4\norm{\delta u}_4\dd s\\
    & \leq  \int_{N^{-2}}^T \norm{\dot{u}_2}^{1\slash4}_2 \norm{\nabla \dot{u}_2}^{3\slash4}_2
    \norm{\delta u}^{1\slash4}_2 \norm{\nabla \delta u}^{3\slash4}_2 \dd s\\
    & \leq  \left( \int_{N^{-2}}^T  \norm{\dot u_2}_2^2 \dd s \right)^\frac{1}{8}  \left( \int_{N^{-2}}^T \norm{\nabla \dot u_2}_2^2  \dd s\right)^\frac{3}{8} \\& \quad \quad \quad \quad\times\left( \int_0^T  \norm{ \delta u}_2^2 \dd s\right)^\frac{1}{8} \left( \int_0^T\norm{\nabla \delta u}_2^2   \dd s\right)^\frac{3}{8} \\
    & \leq C \left( \int_0^T  \sigma(s)^{1-\eta}\norm{\dot u_2}_2^2 \dd s \right)^\frac{1}{8}  \left( \int_0^T \sigma(s)^{2-\eta}\norm{\nabla \dot u_2}_2^2  \dd s\right)^\frac{3}{8}\\&\quad\times \norm{\delta u}_{\infty,2}^{1\slash 4} \norm{\nabla \delta u}_{2,2}^{3\slash 4}.
\end{align*}
In both cases, using Lemma \ref{lem:adddecayprop}, we conclude 
\begin{align*}
   \int_0^T \norm{I_2(s)}_2^2 \dd s < \infty.
\end{align*}
Then, for every $t \in (0,T)$, employing Fubini's Theorem, we rewrite \eqref{eq:domconv} as
\begin{equation}\label{eq:domconv2}
\begin{split}
    \int_0^t \int_{\R^d} & \delta \rho \dot{u}_2 \cdot \delta u \dd x \dd s \\
    & = \lim_{N \to \infty} \lim_{\eps \to 0}  \int_0^t  \int_0^s \int_{\R^d}
    (-\operatorname{div}((\rho_1)_\eps\delta u) + \delta R_\eps)|_{(\tau,X(\tau,x))} 
    (\varphi_N \dot{u}_2 \cdot \delta u)|_{(s,X(s,x))} \dd x \dd \tau \dd s \\
    & =  \lim_{N \to \infty} \lim_{\eps \to 0}  \int_0^t  \int_0^s \int_{\R^d}
    (-\operatorname{div}((\rho_1)_\eps\delta u) + \delta R_\eps)|_{(\tau,x)} 
    (\varphi_N \dot{u}_2 \cdot \delta u)|_{(s,X(s,X^{-1}(\tau,x))} \dd x \dd \tau \dd s \\
    & =: \lim_{N \to \infty} \lim_{\eps \to 0} \; \mathcal{I}^2_{\eps,N}(t) + \mathcal{I}^1_{\eps,N}(t)
\end{split}
\end{equation}
where for every $t \in (0,T)$
\begin{align*}
    & \mathcal{I}^1_{\eps,N}(t) = \int_0^t  \int_0^s \int_{\R^d} \delta R_\varepsilon|_{(\tau,x)}\varphi_N\dot{u}_2\cdot 
    \delta u |_{(s,X(s,X^{-1}(\tau,x))} \dd x \dd \tau \dd s,
    \\& \mathcal{I}^2_{\eps,N}(t)= \int_0^t  \int_0^s \int_{\R^d} -\operatorname{div}((\rho_1)_\eps\delta u) |_{(\tau,x)}
    \varphi_N\dot{u}_2\cdot \delta u |_{(s,X(s,X^{-1}(\tau,x))} \dd x \dd \tau \dd s.
\end{align*}
\textbf{Step 4.}(Commutator estimate) We aim to show that
\begin{align}\label{eq:estimateI1}
    \sup_{t \in (0,T)} \abs{\mathcal{I}^1_{\eps,N}(t)} \to 0 \quad \text{as } \eps \to 0,N \to \infty.
\end{align}
Using \textit{Hölder's inequality} in space and time yields, for every $t \in (0,T)$,
\begin{align*}
    |\mathcal{I}^1_{\eps,N}(t)| & \leq \int_0^T \int_0^s \norm{\delta R_\eps(\tau)}_2
    \norm{(\varphi_N \dot{u}_2 \cdot \delta u)(s)}_2 \dd \tau \dd s\\
    & \leq \int_{N^{-2}}^T \left(\int_0^s \norm{\delta R_\eps(\tau)}_2^2 \dd \tau\right)^{1/2} 
    \sigma(s)^{1/2} \norm{(\dot{u}_2 \cdot \delta u)(s)}_2 \dd s\\
    & \leq  \left(\int_0^T \norm{\delta R_\eps(\tau)}_2^2 \dd \tau \right)^{1/2}
    \int_{N^{-2}}^T \sigma(s)^{1/2} \norm{(\dot{u}_2 \cdot \delta u)(s)}_2 \dd s.
\end{align*}
In order to conclude \eqref{eq:estimateI1}, we claim that for every $\gamma > 0$ there exists an $\eps_0>0$ such that, for 
every $\eps \in (0,\eps_0)$ and every $N \in \N$,
\begin{align*}
    \sup_{t \in (0,T)} \abs{\mathcal{I}^1_{\eps,N}(t)} \leq \gamma.
\end{align*}
Since 
\begin{align*}
    \left(\int_0^T \norm{\delta R_\eps(\tau)}_2^2 \dd \tau \right)^{1/2} \leq 
    \norm{R_{1,\eps}}_{L^2(0,T;L^2)} + \norm{R_{2,\eps}}_{L^2(0,T;L^2)} \underset{\eps \to 0}{\to} 0,
\end{align*}
it suffices to show that 
\begin{align*}
    \int_{N^{-2}}^T \sigma(s)^{1/2} \norm{(\dot{u}_2 \cdot \delta u)(s)}_2 \dd s
\end{align*}
is bounded uniformly in $N.$

In the case $d=2$, using \textit{Gagliardo-Nirenberg inequality}, we have 
\begin{align*}
    \int_{N^{-2}}^T \sigma(s)^\frac{1}{2}\norm{(\dot{u}_2 \cdot \delta u)(s)}_2 \dd s 
    & \leq \int_{N^{-2}}^T \sigma(s)^{-\frac{1}{4}+\frac{\eta}{2}}\sigma(s)^{\frac{1}{4}-\frac{\eta}{4}}
    \sigma(s)^{\frac{1}{2}-\frac{\eta}{4}}\norm{(\dot{u}_2 \cdot \delta u)(s)}_2
    \dd s \\ 
    & \leq \left( \int_{N^{-2}}^T \sigma(s)^{\frac{1}{2}-\frac{\eta}{2}}\sigma(s)^{1-\frac{\eta}{2}} \norm{\dot 
    u_2}_2 \norm{\nabla \dot u_2}_2 \dd s\right)^\frac{1}{2} \\
    &\quad \times\left(\int_{N^{-2}}^T \sigma(s)^{-\frac{1}{2}+\eta} \norm{\delta u}_2 \norm{\nabla \delta u}_2 
    \dd s \right)^\frac{1}{2}\\ 
    & := \mathcal{T}_1^\frac{1}{2}\times\mathcal{T}_2^\frac{1}{2}.
\end{align*}
For the first term, we use \textit{Cauchy-Schwarz inequality} and the bounds of Lemma \ref{lem:adddecayprop} to get
\begin{align*}
    \mathcal{T}_1 \leq \left( \int_{N^{-2}}^T \sigma(s)^{1-\eta}\norm{\dot u_2}^2_2 
    \dd s\right)^\frac{1}{2} \left( \int_{N^{-2}}^T \sigma(s)^{2-\eta} \norm{\nabla \dot u_2}^2_2
    \dd s\right)^\frac{1}{2} \leq C.
\end{align*}
For the second term, using again \textit{Cauchy-Schwarz inequality}, we obtain
\begin{align*}
    \mathcal{T}_2 & \leq \left(\int_{N^{-2}}^T \sigma(s)^{-1+2\eta} \norm{\delta u}^2_2 \dd s\right)^\frac{1}{2} 
    \left(\int_{N^{-2}}^T \norm{\nabla \delta u}^2_2 \dd s\right)^\frac{1}{2}\\ 
    & \leq \norm{\delta u}_{\infty,2} \norm{\nabla \delta u}_{2,2} \left(\int_{N^{-2}}^T \sigma(s)^{-1+2\eta} 
    \dd s\right)^\frac{1}{2} 
\end{align*}
which is uniformly bounded in $N$ as $\eta >0$. In the three-dimensional case, we have    
\begin{align*}
    \int_{N^{-2}}^T \sigma(s)^\frac{1}{2}\norm{(\dot{u}_2 \cdot \delta u)(s)}_2 \dd s & 
    \leq \int_{N^{-2}}^T \sigma(s)^{\frac{1}{8}-\frac{\eta}{8}}\sigma(s)^{\frac{3}{4}-\frac{3\eta}{8}} 
    \sigma(s)^{-\frac{3}{8}+\frac{\eta}{2}}\norm{(\dot{u}_2 \cdot \delta u)(s)}_2 \dd s \\ 
    & \leq \left( \int_{N^{-2}}^T \sigma(s)^{\frac{1}{4}-\frac{\eta}{4}}\sigma(s)^{\frac{3}{2}-\frac{3\eta}{4}} 
    \norm{\dot u_2}^{1\slash2}_2 \norm{\nabla \dot u_2}^{3\slash2}_2 \dd s\right)^\frac{1}{2} \\
    & \quad\times\left(\int_{N^{-2}}^T \sigma(s)^{-\frac{3}{4}+\eta} \norm{\delta u}^{1\slash2}_2 
    \norm{\nabla \delta u}^{3\slash2}_2 \dd s\right)^\frac{1}{2}\\ 
    & := \mathcal{T}_3^\frac{1}{2}  \times T_4^\frac{1}{2}
\end{align*}
For the first term, we use \textit{Hölder's inequality} and the bounds of Lemma \ref{lem:adddecayprop} to get
\begin{align*}
    \mathcal{T}_3 \leq \left( \int_{N^{-2}}^t \sigma(s)^{1-\eta}\norm{\dot u_2}^2_2 
    \dd s\right)^\frac{1}{4} \left( \int_{N^{-2}}^t \sigma(s)^{2-\eta} \norm{\nabla \dot u_2}^2_2
    \dd s\right)^\frac{3}{4} \leq C.
\end{align*}
For the second term, using again \textit{Hölder's inequality}, we get
\begin{align*}
    \mathcal{T}_4 & \leq \left(\int_{N^{-2}}^T \sigma(s)^{-3+4\eta} \norm{\delta u}^2_2 
    \dd s\right)^\frac{1}{4} \left(\int_{N^{-2}}^T \norm{\nabla \delta u}^2_2 \dd s\right)^\frac{3}{4}\\ 
    & \leq \norm{\delta u}_{\infty,2} \norm{\nabla \delta u}_{2,2} \left(\int_{N^{-2}}^t \sigma(s)^{-3+4\eta} \dd 
    s\right)^\frac{1}{2} 
\end{align*}
which is uniformly bounded in $N$ as $\eta >1\slash 2$. The claim is now proved.

\textbf{Step 5:} (Analysis of the source terms). In the following, let $p=4$ if $d=2$ and $p=3$ if $d=3$. Let $q$ be the conjugate index of $p$, i.e. $1/p+1/q=1.$ With the transformation rule and integrating by parts, we have
\begin{align*}
    \begin{split}
    |\mathcal{I}^2_{\eps,N}(t)| \leq & \bigg | \int_0^t \int_0^s \int_{\R^d}((\rho_1)_\eps\delta u)(\tau,x)
    DX(s,X^{-1}(\tau,x))\\
    & \cdot DX^{-1}(\tau,x) \nabla (\varphi_N \dot{u}_2 \cdot \delta u) (s,X(s,X^{-1}(\tau,x))) \dd x \dd \tau \dd s 
    \bigg | \\
    \leq & C(T) \int_0^t \int_0^s \norm{(\rho_1)_\eps\delta u}_p \dd \tau
    \norm{\nabla(\varphi_N \dot{u}_2 \cdot \delta u)(s)}_{q} \dd s  \\
    \leq & C(T) \int_0^t \int_0^s \norm{(\rho_1)_\eps\delta u}_p \dd \tau
    \norm{\nabla(\dot{u}_2 \cdot \delta u)(s)}_q \dd s,
    \end{split}
\end{align*}
where we used Lemma \ref{lem:propertiesflow} which implies that
\begin{align*}
    C(T) := \sup_{\tau,s \in (0,T)}\norm{DX(s,X^{-1}(\tau,x)) DX^{-1}(\tau,x)}_\infty < \infty,
\end{align*}
and the estimate
\begin{equation*}
       \norm{\nabla(\varphi_N \dot{u}_2 \cdot \delta u)(s)}_q \leq \varphi_N(s) \norm{\nabla( \dot{u}_2 \cdot \delta u)(s)}_q \leq \norm{\nabla(\dot{u}_2 \cdot \delta u)(s)}_q.
\end{equation*}

\textbf{Step 6.}(Conclusion) Let $\gamma > 0$ be arbitrary and fix $\eps_0 > 0$ such that $|\mathcal{I}^1_{\eps,N}(t)| \leq \gamma $ for every $\eps \in (0,\eps_0)$, every  $N \in \N$ and every $t \in [0,T].$ By \eqref{eq:domconv2} and the previous estimates, we have
\begin{equation*}
    \abs{\int_0^t \int_{\R^d}  \delta \rho \dot{u}_2 \cdot \delta u \, dx \, ds} 
    \leq \gamma + \lim_{\eps \to 0} C(T) \int_0^t \int_0^s \norm{(\rho_1)_\eps\delta u}_p \dd \tau
    \norm{\nabla(\dot{u}_2 \cdot \delta u)(s)}_q \dd s.
\end{equation*}
Since $(\rho_1)_\varepsilon$ is uniformly bounded and it is converging pointwise to $\rho_1$ (see \eqref{eq:pointwise}),
and since $\gamma$ is arbitrary, for every $t \in [0,T]$, we get
\begin{equation*}
    \abs{\int_0^t \int_{\R^d}  \delta \rho \dot{u}_2 \cdot \delta u \, dx \, ds} 
    \leq C(T) \int_0^t \int_0^s \norm{\rho_1 \delta u}_p \dd \tau \norm{\nabla(\dot{u}_2 \cdot \delta u)(s)}_q \dd s.
\end{equation*}
 Now, we pick $\beta <1 \slash 4$ and  we modify the right-hand side as
\begin{align*}
    \int_0^t \int_0^s  \norm{\rho_1 \delta u}_p  & \tau^\beta \tau^{-\beta} \dd \tau
    \norm{\nabla(\dot{u}_2 \cdot \delta u)(s)}_q \dd s  \\ 
    & \leq \int_0^t \left(\int_0^s \norm{\rho_1 \delta u}^2_p \tau^{-2\beta} \dd \tau \right)^\frac{1}{2} 
    \left(\int_0^s \tau^{2\beta} \dd \tau \right)^\frac{1}{2} \norm{\nabla(\dot{u}_2 \cdot \delta u)(s)}_q \dd s \\ 
    & \leq T^{1/2+\beta} \int_0^t \sigma(s)^{\frac{1}{2}+\beta} \left(\int_0^s \norm{\rho_1 \delta u}^2_p \tau^{-2\beta} 
    \dd \tau \right)^\frac{1}{2} \norm{\nabla(\dot{u}_2 \cdot \delta u)(s)}_q \dd s 
\end{align*}
which concludes the proof of Proposition \ref{lem:wminus4stab}
\end{proof}

\subsection{Proof of Theorem \ref{thm:weakstrong1}}
    Fix some $T \in (0,T^*).$ Thanks to Lemma \ref{lem:badgronwall},  for every $t \in (0,T]$, we have
    \begin{align}
     \nonumber   \frac{1}{2}\norm{\sqrt{\rho_1(t)} \delta u(t)}_2^2 + \nu \int_0^t \norm{\nabla \delta u}^2_2 \dd s 
        &= - \int_0^t \int_{\R^d} \delta \rho \dot{ u_2} \cdot \delta u \dd x \dd s \\
        & \quad \quad- \int_0^t \int_{\R^d} \rho_1 \delta u \otimes \delta u : \nabla u_2 \dd x \dd s\\
        &=: I_4(t) + I_3(t) \nonumber
    \end{align}
    where, for every $t \in (0,T)$, 
    \begin{align*}
        I_3(t) := - \int_0^t \int_{\R^d} \rho_1 \delta u \otimes \delta u : \nabla u_2 \dd x \dd s \quad \text{ and }
        \quad I_4(t) := - \int_0^t \int_{\R^d} \delta \rho \dot{ u_2} \cdot \delta u \dd x \dd s.
    \end{align*}
    With Lemma \ref{lem:adddecayprop}, we have, for every $t \in (0,T)$,
    \begin{align*}
        I_3(t) & \leq \int_0^t \norm{\nabla u_2(s)}_\infty \norm{\sqrt{\rho_1(s)} \delta u(s)}_2^2 \dd s.
    \end{align*}
   To estimate the second term, we separate the cases $d=2$ and $d=3$. For clarity, we omit to compute the constants appearing on the right-hand side. Knowing their exact value is not necessary to conclude our weak-strong uniqueness result.
\subsubsection{Proof in the 2d case}
    Using Proposition \ref{lem:wminus4stab} and the \textit{Gagliardo-Nirenberg inequality} 
    $\norm{f}^2_4\leq \norm{\nabla f}_2\norm{f}_2$, for every $t \in (0,T)$, we get
    \begin{align*}
        \abs{I_4(t)} & \leq   
        \int_0^t \left( \int_0^s \tau^{-2 \beta} \norm{\rho_1(\tau) \delta u(\tau)}_4^2 \dd \tau \right)^{1/2}
        \sigma(s)^{1/2+\beta} \norm{\nabla (\dot u_2\cdot \delta u)}_{4/3} \dd s\\ &\leq  
        \left( \int_0^t \sigma(s)^{-2 \beta} \norm{\rho_1(\tau) \delta u(s)}_4^2 \dd s \right)^{1/2}
        \int_0^t \sigma(s)^{1/2+\beta} \norm{\nabla (\dot u_2\cdot \delta u)}_{4/3} \dd s\\
        & \leq \left( \int_0^t \sigma(s)^{-4 \beta} \norm{\delta u(s)}_2^2 \dd s \right)^{1/4}
        \int_0^t \sigma(s)^{1/2+\beta} \norm{\nabla (\dot u_2\cdot \delta u)}_{4/3} \dd s
        \norm{\nabla \delta u}_{2,2}^{1/2}\\ 
        & \leq \left( \int_0^t \sigma(s)^{-4 \beta} \norm{\rho_1(\tau) \delta u(s)}_2^2 \dd s \right)^{1/4}
        \int_0^t \sigma(s)^{1/2+\beta} \norm{\nabla (\dot u_2\cdot \delta u)}_{4/3} \dd s 
        \norm{\nabla \delta u}_{2,2}^{1/2},
    \end{align*}
where, in the last step, we used that the density is bounded from below. Then, using that
    \begin{align*}
        \norm{\nabla (\dot u_2\cdot \delta u)}_{4/3} &= \norm{\nabla \dot u_2 \delta u + 
        \nabla \delta u\, \dot{u}_2}_{4/3} \\&\leq \norm{\nabla \dot u_2}_2 \norm{\delta u}_2^{1/2} 
        \norm{\nabla \delta u}^{1/2}_{L^2} + \norm{\nabla \delta u}_2 \norm{\dot u_2}_2^{1/2} 
        \norm{\nabla \dot u_2}^{1/2}_{L^2},
    \end{align*}
    we have
    \begin{align*}
        \int_0^t \sigma(s)^{1/2+\beta} \norm{\nabla (\dot u_2\cdot \delta u)}_{4/3} 
        \dd s \leq & \int_0^t \sigma(s)^{1/2+\beta+\alpha} 
        \norm{\nabla \dot u_2}_2 \sigma(s)^{-\alpha}\norm{\delta u}_2^{1/2} 
        \norm{\nabla \delta u}^{1/2} \dd s \\
        & + \int_0^t \sigma(s)^{1/2+\beta-\eps} \norm{\nabla \delta u}_2 
        \norm{\dot u_2}_2^{1/2} \norm{\nabla \dot u_2}^{1/2}  \dd s\\ 
        \leq & \left(\int_0^t \sigma(s)^{1+2\beta+2\alpha} \norm{\nabla \dot 
        u_2}_2^2 \dd s \right)^{\frac 14}
        \left( \int_0^t \sigma(s)^{-4\alpha} \norm{\delta u}_2^2 \dd s \right)^{\frac 14}
       \|\nabla \delta u\|_{2,2}^{\frac12}\\
         &+  \left( \int_0^t \sigma(s)^{4 \beta_1} \norm{\dot u_2}_2^2 \dd s \right)^{\frac14}
        \left( \int_0^t \sigma(s)^{4 \beta_2} \norm{\nabla \dot u_2}_2^2 \dd s \right)^{\frac14}
        \|\nabla \delta u\|_{2,2},
    \end{align*}
    for a suitable $\alpha \in (0,1/4),$ $\eps > 0$, and $\beta_1,\beta_2>0$ verifying $\beta_1 + \beta_2 = 1/2+\beta-\eps.$
    We choose the parameters such that
    \begin{align*}
        4 \beta_1=1-\eta,\quad 4 \beta_2=2-\eta,\quad \alpha = \beta = \frac{1}{4} - \frac{\eta}{4} \quad \text{and} \quad \eps = \frac{\eta}{4}.
    \end{align*}
    Lemma \ref{lem:adddecayprop} implies that there is $C=C(T)>0$ such that,   for every $t \in (0,T]$,
    \begin{align*}
        \int_0^t \sigma(s)^{1 - \eta} \norm{\dot u_2}_2^2 \dd s +
        \int_0^t \sigma(s)^{2 - \eta} \norm{\nabla \dot u_2}_2^2 \leq C.
    \end{align*}
We ignore the factor $t^\eps$ which is bounded on $(0,T).$ Hence
    \begin{align*} 
        \int_0^t \sigma(s)^{1/2+\beta} \norm{\nabla (\dot u_2 \cdot \delta u)}_{4/3} \dd s &
        \leq  \left( \int_0^t \sigma(s)^{-(1-\eta)} \norm{\delta u}_2^2 \dd s \right)^{\frac{1}{4}}
        \norm{\nabla \delta u}_{2,2}^{1/2} + \norm{\nabla \delta u}_{2,2}\\
        & \leq  \left( \int_0^t \sigma(s)^{-(1-\eta)} \norm{\sqrt{\rho_1} \delta u}_{2}^2 \dd s \right)^{\frac{1}{4}}
        \norm{\nabla \delta u}_{2,2}^{1/2} + \norm{\nabla \delta u}_{2,2},
    \end{align*}
    where we used that $\rho_1^{-1} \in L^\infty(\R^d)$ in the last
    line. Gathering the above estimates, for almost every $t \in (0,T]$, we have
    \begin{align*}
        |I_4(t)| \leq & \left(\int_0^t \sigma(s)^{-(1-\eta)} \norm{\sqrt{\rho_1}\delta u}_2^2 \dd s \right)^{\frac{1}{2}} 
        \norm{\nabla \delta u}_{2,2}^{\frac{1}{2}}\\
        & + \left(\int_0^t \sigma(s)^{-(1-\eta)} \norm{\sqrt{\rho_1} \delta u}_2^2 \dd s \right)^{\frac{1}{4}} 
        \norm{\nabla \delta u}_{2,2}^{\frac{3}{2}}\\
        \leq & \int_0^t \sigma(s)^{-(1-\eta)} \norm{\sqrt{\rho_1}\delta u}_2^2 \dd s
        + \frac{\nu}{2} \int_0^t \norm{\nabla \delta u}_{2,2}^2 \dd s
    \end{align*}
  The bounds on $I_3$ and $I_4$ yield, for almost every $t \in (0,T]$, 
    \begin{align*}
        \frac{1}{2}\norm{\sqrt{\rho_1}\delta u}_2^2 + \frac{\nu}{2}
        \int_0^t \norm{\nabla \delta u}^2_2 \dd s \leq \int_0^t \left(
        \norm{\nabla u}_\infty + \sigma(s)^{-(1-\eta)} \right) 
        \norm{\sqrt{\rho_1}\delta u}_2^2 \dd s.
    \end{align*}
    Since
    \begin{align*}
        t \mapsto \norm{\nabla u(t)}_\infty + \sigma(t)^{-(1-\eta)} \in L^1(0,T),
    \end{align*}
    we deduce with \textit{Grönwall's inequality} that, for almost every $t \in (0,T)$,
    \begin{align*}
        \norm{\delta u(t)}_2^2 \leq C \norm{\sqrt{\rho_1(t)}
        \delta u(t)}_2^2 \leq  C \norm{\sqrt{\rho_1(0)}
        \delta u(0)}_2^2 =  0,
    \end{align*}
    which implies that $u_1=u_2$ almost everywhere in
    $[0,T] \times \R^2$. Since $T>0$ was arbitrary, we have that $u_1=u_2$ 
    almost everywhere in $[0,T^*) \times \R^2$.
\qed
\subsubsection{Proof in the 3d case}
The proof follows the 2d case but with the embedding $\norm{f}_4\leq \norm{\nabla f}^{3\slash4}_2\norm{f}^{1\slash4}_2$.
We have
    \begin{align*}
        \abs{I_4(t)} \leq &  
        \left( \int_0^t \sigma(s)^{-2 \beta} \norm{\rho_1(\tau) \delta u(s)}_3^2 \dd s \right)^{1/2}
        C \int_0^t \sigma(s)^{1/2+\beta} \norm{\nabla (\dot u_2\cdot \delta u)}_{3/2} \dd s\\  
        \leq & \left( \int_0^t \sigma(s)^{-2 \beta} \norm{\rho_1 \delta u(s)}_2 \norm{\delta u(s)}_6 \dd s \right)^{1/2}
        \int_0^t \sigma(s)^{1/2+\beta} \norm{\nabla (\dot u_2\cdot \delta u)}_{3/2} \dd s\\
        \leq & \left( \int_0^t \sigma(s)^{-4 \beta} \norm{\rho_1 \delta u(s)}_2^2 \dd s \right)^{1/4}
        \left( \int_0^t \norm{\nabla \delta u(s)}_2^2 \dd s \right)^{1/4}
        \int_0^t \sigma(s)^{1/2+\beta} \norm{\nabla (\dot u_2\cdot \delta u)}_{3/2} \dd s,
    \end{align*}
where, in the second step, we used the interpolation inequality $\norm{f}_3 \leq \norm{f}_2^{1/2} \norm{f}_6^{1/2}$, in the third step, that $\rho_1$ is bounded and, in the last step, \textit{Hölder's inequality}.

Using the product rule and \textit{Sobolev's and Hölder's inequality} several times yield
\begin{align*}
    \norm{\nabla (\dot u_2\cdot \delta u)}_{3/2} \leq \norm{\nabla \dot u_2}_2 \norm{\delta u}_6
    + \norm{\dot u_2}_6 \norm{\nabla \delta u}_2 \leq \norm{\nabla \dot u_2}_2 \norm{\nabla \delta u}_2.
\end{align*}
Hence 
\begin{align*}
    \int_0^t \sigma(s)^{1/2+\beta} \norm{\nabla (\dot u_2\cdot \delta u)}_{3/2} \dd s
    & \leq \int_0^t \sigma(s)^{1/2+\beta} \norm{\nabla \dot u_2}_2 \norm{\nabla \delta u}_2 \dd s \\
    & \leq \left(\int_0^t \sigma(s)^{1+2 \beta} \norm{\nabla \dot u_2}_2 \dd s \right)^{1/2} 
    \left( \int_0^t \norm{\nabla \delta u}_2^2 \dd s \right)^{1/2}. 
\end{align*}
Choosing $\beta = \max(0,1/2-\eta/2)$, we have
\begin{align*}
    \beta \in [0,1/4) \quad \Leftrightarrow \quad 1/2-\eta/2 < 1/4 \quad \Leftrightarrow \quad
    1/2 < \eta.
\end{align*}
We conclude that there exists a $C=C(T)>0$ such that
\begin{align*}
    \int_0^t \sigma(s)^{1-\eta/2} \norm{\nabla (\dot u_2\cdot \delta u)}_{3/2} \dd s \leq 
    C \left( \int_0^t \norm{\nabla \delta u}_2^2 \dd s \right)^{1/2},
\end{align*}
which implies
\begin{align*}
    \abs{I_4(t)} \leq & \left( \int_0^t \sigma(s)^{-2+2\eta} \norm{\rho_1 \delta u(s)}_2^2 \dd s \right)^{1/4}
    \left( \int_0^t \norm{\nabla \delta u(s)}_2^2 \dd s \right)^{3/4} \\
    \leq & C \int_0^t \sigma(s)^{-2+2\eta} \norm{\rho_1 \delta u(s)}_2^2 \dd s 
    + \frac{\nu}{2} \int_0^t \norm{\nabla \delta u(s)}_2^2 \dd s.
\end{align*}
The rest of the proof follows exactly the $\R^2$ case. This concludes the proof of Theorem \ref{thm:weakstrong1}.
\qed
\vspace{1cm}
\noindent\textbf{Acknowledgements}
\medbreak
\noindent The authors express their gratitude to Gianluca Crippa, Raphaël Danchin and Emil Wiedemann for valuable discussions and suggestions on an earlier version of the manuscript.

\vspace{0.5cm}
\noindent\textbf{Fundings}
\medbreak
\noindent T. Crin-Barat is supported by the Alexander von Humboldt-Professorship program and the Deutsche
Forschungsgemeinschaft (DFG, German Research Foundation) under project C07 of the
Sonderforschungsbereich/Transregio 154 ``Mathematical Modelling, Simulation and
Optimization using the Example of Gas Networks" (project ID: 239904186). S. \v{S}kondri\'c and A. Violini are supported 
by the Deutsche Forschungsgemeinschaft project “Inhomogeneous and compressible fluids: statistical solutions and 
dissipative anomalies” within SPP 2410 Hyperbolic Balance Laws in Fluid Mechanics: Complexity, Scales, Randomness 
(CoScaRa).

Part of this work was done during A. Violini's visit to the Chair of Analysis at the Friedrich-Alexander-Universität Erlangen-N\"urnberg (FAU) and S. \v{S}kondri\'c's visit to the Department of Mathematics and Computer Science at the University of Basel.

\appendix

\section{Decomposition of Leray-Hopf solutions}

In this section, we show that the velocity field $u$ of a 
Leray-Hopf solution satisfies
\begin{align*}
    \frac{u}{1+|x|} \in L^1_\loc (0,T;L^1(\R^d) + L^\infty(\R^d)).
\end{align*}
This allows us to apply the DiPerna-Lions theory to study the evolution of the density in the continuity equation. We begin with the following decomposition result which is a direct consequence of \cite[Theorem 2.2]{OrtnerSuli2012}.

\begin{proposition}[\cite{OrtnerSuli2012}]\label{prop:Ortner}
    Let $u \in H^1_\loc(\R^d; \R^d)$ with $\nabla u \in L^2(\R^d),$ then $u$ can be decomposed in 
    $u= v+w$ with
    \begin{align*}
        \norm{\nabla v}_{L^\infty} \leq C \norm{\nabla u}_{L^2} \quad \text{and} \quad
        \norm{w}_{H^1}  \leq C \norm{\nabla u}_{L^2}.
    \end{align*}
\end{proposition}
\begin{proof}
The key idea to prove Proposition \ref{prop:Ortner} is to consider a space mollifier $\varphi$ and set
\begin{align*}
v = u * \varphi \quad \text{and} \quad w = u-v.
\end{align*}
Then the estimates for $v$ and $w$ follow from \textit{Hölder's inequality} and \textit{Poincaré's inequality}.
\end{proof}

Tracking carefully the time dependence, we obtain the following decomposition result for Leray-Hopf solutions.
\begin{lemma}\label{lem:decomptimedep}
    Let $(\rho,u)$ be a Leray-Hopf solution of \eqref{eq:ns} in the sense of Definition \ref{def:lerayhopfsol}, i.e. we
    have 
    \begin{align*}
        u \in L_\loc^2((0,T)\times \R^d) \quad \text{ and } \quad \nabla u \in L^2(0,T;L^2(\R^d)).
    \end{align*}
  As in the proof of Proposition \ref{prop:Ortner}, we get
    \begin{align*}
        \int_0^T \norm{\nabla v}^2_{L^\infty} \dd t \quad\text{and} \quad \int_0^T \norm{ w}^2_{H^1} \dd t < \infty,
    \end{align*}
    and, in particular, we have
    \begin{align*}
        \frac{u}{1+|x|} \in L^1(0,T;L^1(\R^d) + L^\infty(\R^d)).
    \end{align*}
\end{lemma}

\begin{proof}
   Using Proposition \ref{prop:Ortner} and integrating in the time on $[0,T]$, we have
    \begin{align*}
        \int_0^T \norm{\nabla v(t)}^2_{L^\infty} \dd t \leq C   \norm{\nabla u}^2_{{2,2}}  < \infty \quad\text{and}\quad  \int_0^T\norm{w(t)}^2_{H^1}\dd t \leq C   \norm{\nabla u}^2_{{2,2}} < \infty.
    \end{align*}
The bound on $v$ implies that
    \begin{align*}
        \int_0^T \norm{(1+|x|)^{-1} v}_{L^\infty} \dd t < \infty,
    \end{align*}
and the bound $(1+|x|)^{-1} w \in L^1(0,T;L^1(\R^d) + L^\infty(\R^d))$ follows from the following chain of inclusions
    \begin{align*}
        H^1(\R^d) \subset L^2(\R^d) \subset L^1(\R^d) + L^\infty(\R^d).
    \end{align*}
 
\end{proof}

\section{additional decay estimates for strong solutions}
We present some additional properties satisfied by our strong solution $(\rho,u)$. Since the density $\rho$ stays away from zero, we have that the following quantities are bounded (according to Definition \ref{def:strongsol}):
\begin{equation}\begin{split} \label{eq:StrongEst}
    & \sup_{s \in [0,T]}  \norm{u(s)}^2_{L^2},\:\: \int_0^T \norm{\nabla u(s)}^2  \, ds,
    \\& \sup_{s \in [0,T]} \sigma(s)^{1-\eta} \norm{\nabla u(s)}^2_{L^2}, \quad \int_0^t \sigma(s)^{1-\eta} \left(\norm{\partial_t u(s)}_{L^2}^2+\norm{\nabla^2 u(s)}_{L^2}^2\right)   \, ds,  \\
    &  \sup_{s \in [0,T]} \sigma(s)^{2-\eta} \left( \norm{\partial_t u(s)}^2_{L^2}+\norm{\nabla^2 u(s)}_{L^2}^2\right) \quad \text{and} \quad  \int_0^t\sigma(s)^{2-\eta}\norm{\partial_t \nabla u(s)}_{L^2}^2 \, ds.
\end{split}\end{equation}
Using \textit{Gagliardo-Nirenberg inequality}, we can obtain additional bounds in other norms. Recall that if $d=2$ we require $\eta >0 $, while if $d=3$ assume $\eta >1\slash2 $.

\begin{lemma}\label{lem:decay0}
    Let $d=2,3$ and $(\rho,u) $ be a strong solution in the sense of Definition \ref{def:strongsol}. We have
    \begin{align*}
        \int_0^T \norm{u(s)}^2_{L^\infty} \dd s \quad \text{and} \quad  \int_0^T \sigma(s)^{2-\eta} \norm{\nabla u(s)}^4_{L^4} \dd s.
    \end{align*}
\end{lemma}
\begin{proof}

For $d=2$, using \textit{Gagliardo-Nirenberg inequality}, we get, omitting the time dependence, 
    \begin{align*}
        \int_0^T \norm{u}_{L^\infty}^2 \dd s \leq   \int_0^T \norm{u}_{L^2} \norm{\nabla^2 u}_{L^2} \dd s& \leq \left(\int_0^T \sigma(s)^{\eta-1}\norm{u}^2_{L^2} \dd s \right)^{\frac 12}\left(\int_0^T \sigma(s)^{1-\eta}\norm{\nabla^2 u}^2_{L^2}  \dd s \right)^{\frac 12} \\ & \leq \norm{u}_{\infty,2} \left(\int_0^T \sigma(s)^{\eta-1}\right)^{\frac 12}  \left(\int_0^T \sigma(s)^{1-\eta}\norm{\nabla^2 u}^2_{L^2}  \dd s \right)^{\frac 12},
    \end{align*}
    which is bounded thanks to $\eta >0$ and \eqref{eq:StrongEst}. For the second term, again by \textit{Gagliardo-Nirenberg inequality} and \eqref{eq:StrongEst},  we have
    \begin{align*}
        \int_0^T \sigma(s)^{2-\eta} \norm{\nabla u(s)}^4_{L^4} \dd s 
        & \leq \int_0^T \sigma(s)^{2-\eta}  \norm{\nabla u}_{L^2}^2\norm{\nabla^2 u}_{L^2}^2 \dd s \\&\leq  \left(\sup_{s \in [0,T]} \sigma(s)^{2-\eta} \norm{ \nabla^2 u}_{L^2}^2 \right) \norm{\nabla u}_{L^2L^2}^2 \\&\leq C.
    \end{align*}
    This concludes the proof in the case $d=2$.
For $d=3$, using \textit{Agmon's inequality}, we have $\norm{u}^2_{L^\infty} \leq C\norm{u}_{H^1} \norm{u}_{H^2}$ and analyzing only the term of highest order, we have   
     \begin{align*}
        \int_0^T \norm{\nabla u}_2\norm{\nabla^2 u}_2 \dd s & \leq \left(\int_0^T \sigma(s)^{\eta-1}\norm{\nabla u}^2_{L^2} \dd s \right)^{\frac12}\left(\int_0^T \sigma(s)^{1-\eta}\norm{\nabla^2 u}^2_{L^2}  \dd s \right)^{\frac12} \\ & \leq  \left(  \sup_{s\in [0,T]} \sigma(s)^{1-\eta} \norm{\nabla u}^2_2\right)^{\frac12} \left(\int_0^T \sigma(s)^{2\eta-2} \dd s \right)^{\frac12} \left(\int_0^T \sigma(s)^{1-\eta}\norm{\nabla^2 u}^2_2  \dd s \right)^{\frac12},
    \end{align*}
    which is finite thanks to $\eta>1\slash 2$ and \eqref{eq:StrongEst}. The last estimate is obtained using \textit{Gagliardo-Nirenberg inequality} as follows
    \begin{align*}
        \int_0^t \sigma(s)^{2-\eta} \norm{ \nabla u}^4_{L^4}  \dd s & \leq C     \int_0^t \sigma(s)^{2-\eta} \norm{ \nabla u }_{L^2} \norm{ \nabla^2 u }^3_{L^2} \dd s \\ & \leq  C \norm{\nabla u}_{4,2} \left( \int_0^t \sigma(s)^{8\slash 3-4\eta\slash 3}  \norm{ \nabla^2 u }^4_{L^2} \dd s \right)^{\frac34}
        \\ & \leq  C \norm{\nabla u}_{4,2}  \left( \sup_{s \in [0,t]} \sigma(s)^{2-\eta} \norm{\nabla^2 u}^2_{L^2} \right)^{\frac32} \left(\int_0^t \sigma(s)^{-4\slash 3+2\eta\slash 3}   \dd s \right)^{\frac34} \\&\leq C,
    \end{align*}
    where, on the last line, we used $-4\slash 3+2\eta\slash 3 > -1$ and the fact that by \eqref{eq:StrongEst} and $\eta >1\slash 2$, we have 
    \begin{align*}
       \int_0^t \norm{ \nabla u}_2^4 \dd s \leq \left( \sup_{s \in [0,t]} \sigma(s)^{1-\eta} \norm{\nabla u}^2_{L^2} \right)^2 \int_0^t \sigma(s)^{2\eta-2} \dd s < \infty.
    \end{align*}
\end{proof}
Lemma \ref{lem:decay0} allows us to estimate the material derivative
$\dot u = \partial_t u + (u \cdot \nabla) u$ of the velocity of our strong solution.

\begin{lemma}[Additional decay properties]\label{lem:adddecayprop}
    Let $(\rho,u) $ a strong solution in the sense of Definition \ref{def:strongsol}. For $d=2,3$, we have
    \begin{align*}
        \int_0^T \sigma(s)^{1-\eta}\norm{\dot u(s)}_{L^2}^2 + \sigma(s)^{2-\eta} \norm{\nabla \dot u(s)}_{L^2}^2 \dd s \leq C.
    \end{align*}
    Moreover, since $(\rho,u)$ solves \eqref{eq:ns} as an identity in $L^2$, we have
    \begin{align}\label{eq:lip}
        \int_0^T \norm{\nabla u(s)}_{L^\infty} \dd s \leq C.
    \end{align}
\end{lemma}

\begin{proof}
   Using Lemma \ref{lem:decay0} and \eqref{eq:StrongEst} we have 
    \begin{align*}
        \int_0^T \sigma(s)^{1-\eta} \norm{\dot u}_{L^2}^2 \dd s & \leq \int_0^T \sigma(s)^{1-\eta} \left(\norm{\partial_t  u}_{L^2}^2 + \norm{ (u \cdot \nabla) u }_{L^2}^2 \right)\dd s \\
        & \leq \int_0^T \sigma(s)^{1-\eta} \left(\norm{\partial_t  u}_{L^2}^2 + \norm{u }_{L^\infty}^2 \norm{ \nabla u }_{L^2}^2 \right)\dd s \leq C.
    \end{align*}
    For the gradient of the material derivative, using again Lemma \ref{lem:decay0} and \eqref{eq:StrongEst}, we get
   \begin{align*}
        \int_0^T \sigma(s)^{2-\eta} \norm{\nabla\dot u}_{L^2}^2 \dd s & \leq \int_0^T \sigma(s)^{2-\eta} \left(\norm{\partial_t \nabla u}_{L^2}^2 + \norm{\nabla (u \cdot \nabla) u }_{L^2}^2 \right)\dd s \\
        & \leq \int_0^T \sigma(s)^{2-\eta} \left(\norm{\partial_t \nabla u}_{L^2}^2 + \norm{\nabla u \nabla u}_{L^2}^2+\norm{ u \nabla^2 u}_{L^2}^2\right) \dd s \\
        & \leq \int_0^T \sigma(s)^{2-\eta} \left(\norm{\partial_t \nabla u}_{L^2}^2 + \norm{\nabla u}_{L^4}^4+\norm{ u}_{L^\infty}^2\norm{ \nabla^2 u}_{L^2}^2\right) \dd s \leq C.
    \end{align*}
    To derive \eqref{eq:lip}, for $d=2$, the proof can be found in \cite[Lemma 3.2]{PaicuZhangZhang2013} and, for $d=3$, it comes from a direct adaptation of \cite[Lemma 3.2]{PaicuZhangZhang2013}.
\end{proof}

\vspace{3mm}
\bibliographystyle{abbrv}
\bibliography{Inhomogeneous-ref.bib}

\vfill 

\end{document}